\documentclass{amsart}
\usepackage{amsfonts,xcolor}
\usepackage{amsthm}
\usepackage{latexsym}
\usepackage{amsmath}
\usepackage{amssymb}
\usepackage{arydshln}
\newtheorem{theorem}{Theorem}
\theoremstyle{plain}

\newtheorem{corollary}{Corollary}

\newtheorem{lemma}{Lemma}

\newtheorem{proposition}{Proposition}
\newtheorem{remark}{Remark}

\numberwithin{equation}{section}

\begin{document}
\title{On the irrationality exponent of the regular paperfolding numbers}
\author{Yingjun Guo, Wen Wu and Zhixiong Wen}
%\address[A. One and A. Two]{Author OneTwo address line 1\\
%Author OneTwo address line 2}
%\email[A. One]{aone@aoneinst.edu}
%\urladdr{http://www.authorone.oneuniv.edu}

\begin{abstract}
In this paper, improving the method of Allouche \emph{et al.} \cite{APWW98}, we calculate the Hankel determinant of the regular paperfolding sequence, and prove that the Hankel determinant sequence module 2 is periodic with period 10 which answers Coon's conjecture \cite{CV12}. Then we extend Bugeaud's method \cite{Bugeaud11} to obatin the exact value of the irrationality exponent for some general transcendental numbers. Using the results above, we prove that the irrationality exponents of the regular paperfolding numbers are exactly 2.
\end{abstract}
%\begin{keyword}
%Paperfolding sequence \sep   Hankel determinant \sep Irrationality exponent \sep  Automatic sequence
%\end{keyword}

\maketitle

%\address[A. One and A. Two]{Author OneTwo address line 1\\
%Author OneTwo address line 2}
%\email[A. One]{aone@aoneinst.edu}
%\urladdr{http://www.authorone.oneuniv.edu}

\section{Introduction}

%The history of Diophantine approximation is quite old, one of its goal is to compare the distance between a given real number $\xi$ and a rational
%number $p/q$ with the denominator $q $. The irrationality exponent measures the quality of the best rational approximations to $\xi$. Let recall the %definition of irrationality exponent.
Let $\xi $ be an irrational real number, the \emph{irrationality exponent} $(
$or \emph{irrationality measure}$)$ of $\xi $, denoted by $\mu (\xi )$, is
defined as the supremum of the set of real numbers $\mu$ such that the
inequality
\begin{eqnarray*}
|\xi-\frac{p}{q}|<\frac{1}{q^{\mu}}
\end{eqnarray*}
has infinitely many solutions $(p,q)\in\mathbb{Z}\times\mathbb{N}$.

An application of the theory of continued fraction, we know that $\mu (\xi
)\geq 2$ for all irrational number $\xi$. Khintchine's Theorem \cite{K64}
tells us that $\mu (\xi )=2$ for Lebesgue-almost all real numbers $\xi $.
Furthermore, Roth's Theorem \cite{R55} asserts that $\mu (\xi )=2$ for every
algebraic irrational number.

For transcendental numbers whose continued fraction expansions are known,
one can get the exact value of the irrationality exponents. For example, the
irrationality exponent of $e$ and the sturm number are given (see \cite{A10}%
, \cite{BKS11}, \cite{L02}). While we do not know the continued fraction
expansions of transcendental number, we hardly get the exact values of their
irrationality exponents. There are two examples. One is the classical transcendental numbers $\pi,\ln(2)$ (see \cite{S08},\cite{N10} and therein), the other is  the \emph{automatic number} \cite{AS03} defined by their expansion in some integer.
%that we only have the upper bound of $\mu
%(\pi),\mu (\ln(2))$ (see \cite{S08},\cite{N10} and there in). The other example is the \emph{automatic number} \cite{AS03} defined by their expansion in some integer.
% A class of real numbers defined by their expansion in some integer, called the \emph{automatic number} \cite{AS03}, have been studied by many authors.
%\textcolor[rgb]{0.00,0.00,1.00}{There is a class of real numbers defined by their expansion in some integer and whose continued fraction expansion is unknown.}
%\textcolor[rgb]{1.00,0.00,0.00}{And, many authors have studied a
%special class of numbers, called \emph{automatic number} \cite{AS03}, although their
%continued fraction expansion are not clear.}
Recall that a real number $\xi $ with $b$-ary expansion
\begin{equation*}
\xi =\sum_{i\geq 0}u_{i}\frac{1}{b^{i}}
\end{equation*}%
is called an automatic real number if $\{u_{i}\}_{i\geq 0}$ is an automatic
sequence. In fact, all the irrational automatic real numbers are
transcendental, this result was proved by Adamczewski and Bugeaud \cite{AY07}%
. %2007, Adamczewski and Bugeaud \cite{AY07} proved that all the irrational
%automatic real numbers are transcendental.
In 2006, Adamczewski and Cassaigne \cite{AC06} proved that all automatic
real numbers have finite irrationality exponents. In 2008, Bugeaud \cite%
{Bugeaud08} constructed a class of real numbers whose irrationality exponent
can be read off from their $b$-ary expansion and proved that there exist
automatic real numbers with any prescribed rational irrationality exponent.
In 2011, Bugeaud, Krieger and Shallit \cite{BKS11} showed that the
irrationality exponent of every automatic (resp. morphic) number in that
class is rational (resp. algebraic). And they conjectured that the result
remains true for all automatic (resp. morphic) number. In 2011, applying the
fact that the Hankel determinants of the Thue-Morse sequence over $\{-1,1\}$
are nonzero \cite{APWW98}, Bugeaud \cite{Bugeaud11} proved that the
irrationality exponents of the Thue-Morse real numbers are exactly 2.
%In 1998, Allouche, Peyri\`{e}re, Wen and Wen \cite{APWW98} showed that the
%Hankel determinants of the Thue-Morse sequence over $\{-1,1\}$ are nonzero.
%Applying this fact, Bugeaud \cite{Bugeaud11} proved that the irrationality
%exponents of the Thue-Morse real numbers are exactly 2, recently.
Using Bugeaud's method, in 2012, Coons \cite{C12} proved that the
irrationality exponent of the sum of the reciprocals of the Fermat numbers
is 2. Recently, Wen and Wu \cite{WW13} showed that the
irrationality exponents of the Cantor real numbers are exactly 2 in the same way.

In this paper, we extend Bugeaud's method to some general
transcendental numbers and determine the irrationality exponent of the regular paperfolding numbers.
\subsection{The main result}
Let $\{u_{i}\}_{i\geq 0}$ be an integer sequence, whose generating function
is $f(x)=\sum_{i\geq 0}u_{i}x^{i}$. The determinant of the $n\times n$-matrix $(u_{i+j-2})_{1\leq i,j\leq n }$ is called the \emph{Hankel determinant of
order $n$} associated to $f(x)$ (or the sequence $\{u_{i}\}_{i\geq 0}$), denoted by $H_{n}(f)$. Our main result is as follows.

\begin{theorem}
\label{proposition1} Suppose the generating function $f(x)=\frac{A(x)}{B(x)}%
+C(x)f(x^{k})$ $(k\geq 2)$, where $A(x),B(x),C(x)\in \mathbb{Z}[x]$. Let $%
\{n_{i}\}_{i\geq 0}$ be an increasing positive integer sequence. If for all $%
i\geq 0$, $H_{n_{i}}(f)H_{n_{i}+1}(f)\neq 0$, then for all $b\geq 2$,
\begin{equation*}
\mu \left( f(\frac{1}{b})\right) \leq 2k.
\end{equation*}%
Moreover, if $\liminf\limits_{i\rightarrow \infty }\frac{n_{i+1}}{n_{i}}=1,$
then,
\begin{equation*}
\mu \left( f(\frac{1}{b})\right) =2.
\end{equation*}
\end{theorem}

\begin{remark}
While determining the irrationality exponent of the Thue-Morse numbers,
Bugeaud \cite{Bugeaud11} used the fact that all the Hankel determinants of
the Thue-Morse sequence are nonzero \cite{APWW98}. In fact, we only need a
proper subsequence $\{n_{i}\}_{i\geq0}$ satisfying $H_{n_{i}}(f)H_{n_{i}+1}(f)\neq0
$. This requirement seems to be feasible for some transcendental numbers.
\end{remark}
\begin{remark}
$H_{n_{i}}(f)\neq0$ for $i\geq0$ implies that $f(\frac{1}{b})\notin\mathbb{Q}$. Under the assumption of Theorem \ref{proposition1}, $f(\frac{1}{b})$ can be transcendental. For example, the generating functions of the Thue-Morse, the Cantor sequence and the paperfolding sequence, satisfy the assumption of Theorem \ref{proposition1}.
\end{remark}

\subsection{Irrationality exponent of the regular paperfolding sequence }
\iffalse
In the following, we will investigate the irrationality exponent of the
regular paperfolding numbers. We need some definitions. Let $\mathcal{A}$ be
the nonempty finite alphabet set and $\epsilon $ be the empty word. $%
\mathcal{A}^{+}$ and $\mathcal{A}^{\ast }$ stand for the set of nonempty
finite words and the set of finite words respectively. A \emph{morphism}
denoted by $\sigma $,
%(sometimes called a \emph{substitution} over $\mathcal{A}$)
is a map from $\mathcal{A}^{\ast }$ to $\mathcal{A}^{\ast }$ satisfying $%
\sigma (UV)=\sigma (U)\sigma (V)$ for $U,V\in {A}^{\ast }$. If $\sigma (a)$
has the same length of $k$ for any $a\in \mathcal{A}$, we say that $\sigma $
is \emph{$k$-uniform}. As usually, a 1-uniform morphism is called a \emph{%
coding}. If there exists a letter $a$ such that $\sigma (a)=aW$ for some $%
W\in \mathcal{A}^{+}$, and furthermore, $\sigma ^{n}(W)\neq \epsilon $ for
all $n\geq 0$, then the infinite word $\sigma ^{\infty }(a)=aW\sigma
(W)\sigma ^{2}(W)\cdots $ is a \emph{fixed point} of $\sigma $. Such an
infinite fixed points is called a \emph{pure morphic} sequence. If an
infinite word is an image under a coding of a fixed point, we say that it is
\emph{morphic}. Similarly, if an infinite word is an image under a coding of
a fixed point of a uniform morphism, we call that it is \emph{automatic}.
(for details, ref. \cite{AS03}).

The Thue-Morse sequence is a pure 2-automatic sequence over two letters. The
Cantor sequence is a pure 3-automatic sequence over two letters, but it is
non-primitive \cite{AS03}.

\fi
Let $\mathcal{A}=\{a,b,c,d\}$ be a four-letter alphabet. Define the endomorphism $\tau$ on $\mathcal{A}^{*}$ by $\tau: a\mapsto ab, b\mapsto cb, c\mapsto ad, d\mapsto cd$ and the coding $\rho$ from $\mathcal{A}$ to $\{0,1\}$ by $a\mapsto1, b\mapsto1, c\mapsto0, d\mapsto0$. The
\emph{regular paperfolding} sequence is given by%
\begin{equation*}
\mathbf{f:}=f_{0}f_{1}f_{2}\cdots =\rho (\tau ^{\infty
}(a))=110110011100100\cdots
\end{equation*}
Denote the generating function of the regular paperfolding sequence by $F(z):=\sum_{i=0}^{\infty }f_{i}z^{i}$. Let $b\geq 2$ be an integer, then the \emph{regular paperfolding number} is
defined as follow
\begin{equation*}
\xi _{\mathbf{f},b}:=F(\frac{1}{b})=\sum_{i\geq 0}\frac{f_{i}}{b^{i}}=1+\frac{1}{b}+\frac{1%
}{b^{3}}+\frac{1}{b^{4}}+\frac{1}{b^{7}}+\frac{1}{b^{8}}+\cdots.
\end{equation*}%
For details, see \cite{AS03,D82-1,D82-2,D82-3}.
By Theorem \ref{proposition1},  in order to determine $\mu(\xi _{\mathbf{f},b})$, we need to calculate its Hankel determinants $H_{n}(F)$.
%Fortunately, there are some famous automatic sequences and their corresponding Hankel determinant have been studied. They are the Thue-Morse sequence, the Cantor sequence and the paperfolding sequence.

In 1998, Allouche, Peyri\`{e}re, Wen and Wen \cite{APWW98} studied the Hankel determinants of the Thue-Morse sequence, which is the fixed point of the endomorphism $\sigma: 1\mapsto 1-1, -1\mapsto -11$.
%$\sigma$ over the alphabet $\{-1,1\}$, where $\sigma$ is defined as follow $\sigma: 1\mapsto 1-1, -1\mapsto -11$.
And they proved that the Hankel determinants of the Thue-Morse sequence are all nonzero.
In 2012, with the help of C++ program, Coons and Vrbik \cite{CV12} showed
that the Hankel determinant  $H_{n}(F)\neq0$  of the regular paperfolding sequence for $n\leq2^{13}+3$.
In 2013, Wen and Wu \cite{WW13} investigated the Hankel determinants of the Cantor sequence which is the fixed point of the endomorphism $\theta: 1\mapsto 101, 0\mapsto 000$.
%$\theta$ over the alphabet $\{0,1\}$, where $\theta$ is defined as follow $\theta: 1\mapsto 101, 0\mapsto 000$.
They also proved that the Hankel determinants of the Thue-Morse sequence are all nonzero.

It is observed that the Thue-Morse sequence and the Cantor sequence are fixed points of endomorphisms, while the paperfolding sequence is given by the image under a coding of a fixed point of a endomorphism. Because of this difference, it seems difficult to determine the Hankel determinant of regular paperfolding sequence.
In section 3, we will show the recurrent equations of the Hankel determinants $H_{n}(F)$%
. And the recurrent equations lead to the fact: for all $i\geq0$,
\begin{equation*}
H_{10i+1}(F)H_{10i+2}(F)\neq 0.  \tag{$\star$}
\end{equation*}

%\textcolor[rgb]{1.00,0.00,0.00}{This sequence $\mathbf{f}$ was studied in \cite{{D82-1}}, \cite{{D82-2}} and \cite{{D82-3}}.}
%Since there is a coding in the construction, the regular paperfolding
%sequence is not a pure morphic sequence \cite{AS03}.

Now, we turn to study the irrationality exponents of the regular paperfolding numbers. Before our study, there are some results already.  In 2009, Adamczewski and Rivoal \cite{AR09} proved
that $\mu (\xi _{\mathbf{f},b})\leq 5$. Then, in 2012, Coons and Vrbik \cite{AR09}
improved this estimate by the inequality $\mu (\xi _{\mathbf{f},b})\leq2.002075359\cdots $. Here,
we give the exact value of $\mu (\xi _{\mathbf{f},b})$.

\begin{corollary}
\label{THM1} For any integer $b\geq2$, $\mu(\xi_{\mathbf{f},b})=2$.
\end{corollary}

\begin{proof}
Note that the generating function $F(z)$ of the regular paperfolding
sequence $\mathbf{f}$ satisfies
\begin{equation*}
F(z)=\frac{1}{1-z^{4}}+zF(z^{2}).
\end{equation*}%
The result follows from Theorem \ref{proposition1} and the formula $(\star)$.
\end{proof}

The organization of the paper is as follows. In section 2, we give some
notations. In section 3, we present the recurrence equations of the Hankel
determinants $H_{n}(F)$. Then we prove the Hankel determinant sequence
(module $2$) is periodic with period 10. In the last section, we prove
Proposition \ref{proposition1}.

\section{Preliminaries}

In this section,  we will give some definitions and notations.

\begin{itemize}
\item Let $\mathbf{u}=\{u_{n}\}_{n\geq 0}$ be a complex number sequence, then the
\emph{$(p;m,n)$-order Hankel matrix} of the sequence $\mathbf{u}$ is defined
as follow
\begin{equation*}
\mathbf{u}_{m,n}^{p}=\left(
\begin{array}{cccc}
u_{p} & u_{p+1} & \cdots & u_{p+n-1} \\
u_{p+1} & u_{p+2} & \cdots & u_{p+n} \\
\vdots & \vdots & \ddots & \vdots \\
u_{p+m-1} & u_{p+m} & \cdots & u_{p+m+n-2}%
\end{array}%
\right) ,
\end{equation*}%
where $m,n\geq 1$ and $p\geq 0$. If $m=n$, the symbol $\mathbf{u}_{n}^{p}$ is always
used to stand for the Hankel matrix for short.

\item For any matrix $M$, its transposed matrix is denoted by $M^{t}$. If
the matrix $M$ is square, then its Hankel determinant is denoted by $|M|$ .

\item The $m\times n$ matrix with all entries equal to 1 (resp. 0) is
denoted by $\mathbf{1}_{m,n}$ (resp. $\mathbf{0}_{m,n}$).

\item For any $n\geq1$, let $\alpha(n)=(\alpha_{0},\alpha_{1},\cdots,%
\alpha_{n-1})$ and $\beta(n)=(\beta_{0},\beta_{1},\cdots,\beta_{n-1})$ be $%
1\times n$-row vectors with $\alpha_{2i}=\beta_{2i+1}=1,\alpha_{2i+1}=%
\beta_{2i}=0(0\leq i\leq \lfloor\frac{n}{2}\rfloor)$ respectively. For
example, $\alpha(5)=(1,0,1,0,1),\beta(4)=(0,1,0,1)$.

\item Define the matrixes $A_{m,n}=(r_{0},r_{1},\cdots,r_{n-1})$ and $%
B_{m,n}=(s_{0},s_{1},\cdots,s_{n-1})$, where $r_{2i}=s_{2i+1}=\alpha^{t}(m)$
and $r_{2i+1}=s_{2i}=\beta^{t}(m)$. If $m=n$, we simply write $A_{n}$ and $%
B_{n}$. For example,
\begin{equation*}
A_{4,3}= \left(
\begin{array}{ccc}
1 & 0 & 1 \\
0 & 1 & 0 \\
1 & 0 & 1 \\
0 & 1 & 0 \\
\end{array}
\right), B_{5}=\left(
\begin{array}{ccccc}
0 & 1 & 0 & 1 & 0 \\
1 & 0 & 1 & 0 & 1 \\
0 & 1 & 0 & 1 & 0 \\
1 & 0 & 1 & 0 & 1 \\
0 & 1 & 0 & 1 & 0 \\
\end{array}
\right).
\end{equation*}

\item Set the following determinant sequences
\begin{equation*}
\begin{array}{ll}
a_{n}:=|\mathbf{f}_{n}^{0}|, & b_{n}:=\left\vert
\begin{matrix}
\mathbf{f}_{n}^{0} & \alpha ^{t}(n) & \beta ^{t}(n) \\
\alpha (n) & 0 & 0 \\
\beta (n) & 0 & 0%
\end{matrix}%
\right\vert, \\
&  \\
c_{n}:= \left\vert
\begin{array}{cc}
\mathbf{f}_{n}^{0} & \alpha ^{t}(n) \\
\alpha (n) & 0%
\end{array}%
\right\vert , & d_{n}:=\left\vert
\begin{array}{cc}
\mathbf{f}_{n}^{0} & \beta ^{t}(n) \\
\beta (n) & 0%
\end{array}%
\right\vert, \\
&  \\
e_{n}:=\left\vert
\begin{array}{cc}
\mathbf{f}_{n}^{0} & \beta ^{t}(n) \\
\alpha (n) & 0%
\end{array}%
\right\vert, & g_{n}:=|\mathbf{f}_{n+1,n}^{0} ~\alpha ^{t}(n+1)|, \\
&  \\
h_{n}:=| \mathbf{f}_{n+1,n}^{0} \beta ^{t}(n+1)| , & x_{n}:=\left\vert
\begin{array}{ccc}
\mathbf{f}_{n+1,n}^{0} & \alpha ^{t}(n+1) & \beta ^{t}(n+1) \\
\alpha (n) & 0 & 0%
\end{array}
\right\vert , \\
\end{array}
\end{equation*}

$~\quad y_{n}:=\left\vert
\begin{array}{ccc}
\mathbf{f}_{n+1,n}^{0} & \alpha ^{t}(n+1) & \beta ^{t}(n+1) \\
\beta (n) & 0 & 0%
\end{array}
\right\vert.$

\item $U(n)=(e_{1},e_{3},\cdots,e_{2[\frac{n+1}{2}]-1},e_{2},e_{4},%
\cdots,e_{2[\frac{n}{2}]})$, where $e_{j}$ is the $j$-th unit column vector
of order $n$, that is, the column vector with $1$ as its $j$-th entry and
zero elsewhere. If no confusion can occur, we simply write $U$.

\item Unless otherwise stated, the symbol $\equiv$ stands for equality
modulo $2$ in the whole paper.
\end{itemize}

\section{Hankel determinants}

In this section, we will give the recurrence equations of the Hankel
determinants $H_{n}(F)$. By these equations, we prove the Hankel determinant
sequence (module $2$) is periodic. The results are described in Lemma \ref%
{lemma1} and Proposition \ref{lemma2} respectively.

Notice that, the regular paperfolding sequence $\mathbf{f}%
=f_{0}f_{1}f_{2}\cdots$ can be generated by the following recurrence formula
\cite{AR09}:
\begin{equation}  \label{formula1}
f_{4n}=1,f_{4n+2}=0,f_{2n+1}=f_{n}.
\end{equation}

By this formula, we have the following important lemma.

\begin{lemma}
\label{lemma1} For any $n\geq1$,
\begin{eqnarray*}
(1)~\qquad
a_{2n}&=&(-1)^{n+1}(b_{n}^{2}+2c_{n}e_{n}+2d_{n}e_{n}-a_{n}^{2})\equiv
a_{n}+b_{n}, \\
(2)~\quad a_{2n+1}&=&(-1)^{n+1}(2x_{n}y_{n}-g_{n}^{2}-h_{n}^{2})\equiv
g_{n}+h_{n}, \\
(3)~\qquad b_{2n}&=&(-1)^{n+1}(c_{n}+2e_{n}+d_{n})^{2}\equiv c_{n}+d_{n}, \\
(4)~\quad b_{2n+1}&=&(-1)^{n+1}2(x_{n}+y_{n})^{2}\equiv 0, \\
(5)~\qquad c_{2n}&=&(-1)^{n}2\big(b_{n}^{2}+(c_{n}+e_{n})(d_{n}+e_{n})\big)%
\equiv 0, \\
(6)~\quad c_{2n+1}&=&(-1)^{n}\big(4x_{n}y_{n}+(g_{n}+h_{n})^{2}\big)\equiv
g_{n}+h_{n}, \\
(7)~\qquad d_{2n}&=&(-1)^{n}\big(%
2b_{n}^{2}+(c_{n}+e_{n})^{2}+(d_{n}+e_{n})^{2}\big)\equiv c_{n}+d_{n}, \\
(8)~\quad d_{2n+1}&=&(-1)^{n}(x_{n}+y_{n})^{2}\equiv x_{n}+y_{n}, \\
(9)~\qquad e_{2n}&=&(-1)^{n}\big(%
b_{n}(c_{n}+d_{n}-2e_{n})+a_{n}(c_{n}+d_{n}+2e_{n})\big) \\
&\equiv & (a_{n}+b_{n})(c_{n}+d_{n}), \\
(10)~\quad e_{2n+1}&=&(-1)^{n}(g_{n}-h_{n})(x_{n}+y_{n})\equiv
(g_{n}+h_{n})(x_{n}+y_{n}), \\
(11)~\qquad g_{2n}&=&(-1)^{n}\big(%
c_{n}y_{n}+e_{n}x_{n}-e_{n}y_{n}-d_{n}x_{n}+a_{n}(g_{n}+h_{n})\big) \\
& \equiv & a_{n}(g_{n}+h_{n})+x_{n}(d_{n}+e_{n})+y_{n}(c_{n}+e_{n}), \\
(12)~\quad g_{2n+1}&=&(-1)^{n+1}\big(%
c_{n+1}y_{n}+e_{n+1}x_{n}-e_{n+1}y_{n}-d_{n+1}x_{n}+a_{n+1}(g_{n}+h_{n})\big)
\\
&\equiv & a_{n+1}(g_{n}+h_{n})+x_{n}(d_{n+1}+e_{n+1})+y_{n}(c_{n+1}+e_{n+1}),
\\
(13)~\qquad h_{2n}&=&(-1)^{n}\big(%
b_{n}(y_{n}-x_{n})+g_{n}(c_{n}+e_{n})+h_{n}(d_{n}+e_{n})\big) \\
&\equiv & g_{n}(c_{n}+e_{n})+h_{n}(d_{n}+e_{n})+b_{n}(x_{n}+y_{n}), \\
(14)~\quad h_{2n+1}&=&(-1)^{n+1}\big(%
b_{n+1}(y_{n}-x_{n})+g_{n}(c_{n+1}+e_{n+1})+h_{n}(d_{n+1}+e_{n+1})\big) \\
&\equiv & g_{n}(c_{n+1}+e_{n+1})+h_{n}(d_{n+1}+e_{n+1})+b_{n+1}(x_{n}+y_{n}),
\\
(15)~\qquad x_{2n}&=&(-1)^{n}\big(%
2b_{n}(y_{n}-x_{n})+(g_{n}+h_{n})(c_{n}+d_{n}+2e_{n})\big) \\
& \equiv &(g_{n}+h_{n})(c_{n}+d_{n}), \\
(16)~\quad x_{2n+1}&=&(-1)^{n+1}\big(%
2b_{n+1}(y_{n}-x_{n})+(g_{n}+h_{n})(c_{n+1}+d_{n+1}+2e_{n+1})\big) \\
&\equiv & (g_{n}+h_{n})(c_{n+1}+d_{n+1}), \\
(17)~\qquad y_{2n}&=&(-1)^{n+1}(x_{n}+y_{n})(c_{n}-d_{n})\equiv
(x_{n}+y_{n})(c_{n}+d_{n}), \\
(18)~\quad y_{2n+1}&=&(-1)^{n}(x_{n}+y_{n})(c_{n+1}-d_{n+1})\equiv
(x_{n}+y_{n})(c_{n+1}+d_{n+1}).
\end{eqnarray*}
\end{lemma}

\begin{proof}
Let $M=(m_{i,j})_{1\leq i,j\leq n}$ be any $n\times n$-matrix. By the definition of the matrix $U$, we can check easily the following formula
\begin{equation}  \label{equation1}
U^{t}MU=\left(
\begin{array}{cc}
(m_{2i-1,2j-1})_{\begin {subarray}{c} 1\leq i\leq\mu \\1\leq j\leq\mu\end
{subarray}} & (m_{2i-1,2j})_{\begin {subarray}{c} 1\leq i\leq\mu \\1\leq
j\leq\nu\end {subarray}} \\
(m_{2i,2j-1})_{\begin {subarray}{c} 1\leq i\leq\nu \\1\leq j\leq\mu\end
{subarray}} & (m_{2i,2j})_{\begin {subarray}{c} 1\leq i\leq\nu \\1\leq
j\leq\nu\end {subarray}} \\
\end{array}%
\right) ,
\end{equation}%
where $\mu =[\frac{1}{2}(n+1)]$ and $\nu =[\frac{1}{2}n]$. Hence, by the
formulae (\ref{formula1}) and (\ref{equation1}),
\begin{equation}
U^{t}\mathbf{f}_{2n}^{0}U=\left(
\begin{array}{cc}
A_{n} & \mathbf{f}_{n}^{0} \\
\mathbf{f}_{n}^{0} & B_{n}%
\end{array}%
\right)  \label{formula2}
\end{equation}%
and
\begin{equation}
U^{t}\mathbf{f}_{2n+1}^{0}U=\left(
\begin{array}{cc}
A_{n+1} & \mathbf{f}_{n+1,n}^{0} \\
\mathbf{f}_{n,n+1}^{0} & B_{n}%
\end{array}%
\right)  \label{formula3}
\end{equation}

(1) ~First, by the formula (\ref{formula2}) and note that $|U||U^{t}|=1$, we have
\begin{eqnarray*}
\nonumber a_{2n}=|\mathbf{f}_{2n}^{0}|=|U^{t}\mathbf{f}_{2n}^{0}U|=\left\vert
\begin{array}{cc}
A_{n} & \mathbf{f}_{n}^{0} \\
\mathbf{f}_{n}^{0} & B_{n}%
\end{array}%
\right\vert .
\end{eqnarray*}

Then, we add four rows and columns which do not change the determinant. By the elementary operation, we will get a new block matrix which have much zero blocks. That is,
\begin{eqnarray*}
\left\vert
\begin{array}{cc}
A_{n} & \mathbf{f}_{n}^{0} \\
\mathbf{f}_{n}^{0} & B_{n}%
\end{array}%
\right\vert &=&\left\vert
\begin{array}{cc;{2pt/2pt}cccc}
A_{n} & \mathbf{f}_{n}^{0} & \mathbf{0}_{n,1} & \mathbf{0}_{n,1} & \mathbf{0}%
_{n,1} & \mathbf{0}_{n,1} \\
\mathbf{f}_{n}^{0} & B_{n} & \mathbf{0}_{n,1} & \mathbf{0}_{n,1} & \mathbf{0}%
_{n,1} & \mathbf{0}_{n,1} \\
\hdotsfor{6}\\
\alpha (n) & \mathbf{0}_{1,n} & 1 & 0 & 0 & 0 \\
\beta (n) & \mathbf{0}_{1,n} & 0 & 1 & 0 & 0 \\
\mathbf{0}_{1,n} & \beta (n) & 0 & 0 & 1 & 0 \\
\mathbf{0}_{1,n} & \alpha (n) & 0 & 0 & 0 & 1%
\end{array}%
\right\vert\\
\nonumber&=&\left\vert
\begin{array}{cc;{2pt/2pt}cccc}
\mathbf{0}_{n} & \mathbf{f}_{n}^{0} & -\alpha (n)^{t} & -\beta (n)^{t} &
\mathbf{0}_{n,1} & \mathbf{0}_{n,1} \\
\mathbf{f}_{n}^{0} & \mathbf{0}_{n} & \mathbf{0}_{n,1} & \mathbf{0}_{n,1} &
-\alpha (n)^{t} & -\beta (n)^{t} \\
\hdotsfor{6}\\
\alpha (n) & \mathbf{0}_{1,n} & 1 & 0 & 0 & 0 \\
\beta (n) & \mathbf{0}_{1,n} & 0 & 1 & 0 & 0 \\
\mathbf{0}_{1,n} & \beta (n) & 0 & 0 & 1 & 0 \\
\mathbf{0}_{1,n} & \alpha (n) & 0 & 0 & 0 & 1%
\end{array}%
\right\vert  \\
\nonumber&=&\left\vert
\begin{array}{ccc;{2pt/2pt}ccc}
\mathbf{0}_{n} & \mathbf{0}_{n,1} & \mathbf{0}_{n,1} & \mathbf{f}_{n}^{0} &
-\alpha (n)^{t} & -\beta (n)^{t} \\
\mathbf{0}_{1,n} & 1 & 0 & \beta (n) & 0 & 0 \\
\mathbf{0}_{1,n} & 0 & 1 & \alpha (n) & 0 & 0 \\
\hdotsfor{6}\\
\mathbf{f}_{n}^{0} & -\alpha (n)^{t} & -\beta (n)^{t} & \mathbf{0}_{n} &
\mathbf{0}_{n,1} & \mathbf{0}_{n,1} \\
\alpha (n) & 0 & 0 & \mathbf{0}_{1,n} & 1 & 0 \\
\beta (n) & 0 & 0 & \mathbf{0}_{1,n} & 0 & 1 \\
\end{array}%
\right\vert . \\
\end{eqnarray*}

At last, by Laplace's expansion along the first $n+2$ rows, we have
\begin{eqnarray*}
a_{2n} &=&(-1)^{n+1}\cdot \Bigg(-|\mathbf{f}_{n}^{0}|^{2}+2\cdot \left\vert
\begin{array}{cc}
\mathbf{f}_{n}^{0} & \alpha ^{t}(n) \\
\alpha (n) & 0 \\
\end{array}%
\right\vert \cdot \left\vert
\begin{array}{cc}
\mathbf{f}_{n}^{0} & \beta ^{t}(n) \\
\alpha (n) & 0 \\
\end{array}%
\right\vert  \\
&&+2\cdot \left\vert
\begin{array}{cc}
\mathbf{f}_{n}^{0} & \beta ^{t}(n) \\
\beta (n) & 0 \\
\end{array}%
\right\vert \cdot \left\vert
\begin{array}{cc}
\mathbf{f}_{n}^{0} & \beta ^{t}(n) \\
\alpha (n) & 0 \\
\end{array}%
\right\vert +\left\vert
\begin{array}{ccc}
\mathbf{f}_{n}^{0} & \alpha ^{t}(n) & \beta ^{t}(n) \\
\alpha (n) & 0 & 0 \\
\beta (n) & 0 & 0 \\
\end{array}%
\right\vert ^{2}\Bigg) \\
&=&(-1)^{n+1}\left(-a_{n}^{2}+2c_{n}e_{n}+2d_{n}e_{n}+b_{n}^{2}\right)\equiv
a_{n}+b_{n}.
\end{eqnarray*}

We will omit the details of the proofs of the other assertions, since they can be proved in the same way: first, we turn the matrix into a block matrix by the formulae (\ref{formula2}) and (\ref{formula3}). Then, adding finite rows and columns and using the elementary operation again and again, we will get more zero blocks. At last, by the Laplace's expansion, we can calculate the determinant.

(2) Similarly, by the formula (\ref{formula3}),
\begin{eqnarray*}
a_{2n+1}&=&|\mathbf{f}_{2n+1}^{0}| = |U^{t}\mathbf{f}_{2n+1}^{0}U| = \left|
\begin{array}{ccc}
A_{n+1} & \mathbf{f}_{n+1,n}^{0} &  \\
\mathbf{f}_{n,n+1}^{0} & B_{n} &  \\
\end{array}
\right| \\
&=& \left|
\begin{array}{ccc;{2pt/2pt}ccc}
\mathbf{0}_{n+1} & \mathbf{0}_{n+1,1} & \mathbf{0}_{n+1,1} & \mathbf{f}%
_{n+1,n}^{0} & -\alpha(n+1)^{t} & -\beta(n+1)^{t} \\
\mathbf{0}_{1,n+1} & 1 & 0 & \beta(n) & 0 & 0 \\
\mathbf{0}_{1,n+1} & 0 & 1 & \alpha(n) & 0 & 0 \\
\hdotsfor{6}\\
\mathbf{f}_{n,n+1}^{0} & -\alpha(n)^{t} & -\beta(n)^{t} & \mathbf{0}_{n} &
\mathbf{0}_{n,1} & \mathbf{0}_{n,1} \\
\alpha(n+1) & 0 & 0 & \mathbf{0}_{1,n} & 1 & 0 \\
\beta(n+1) & 0 & 0 & \mathbf{0}_{1,n} & 0 & 1 \\
\end{array}
\right| \\
&=& (-1)^{n+1}\cdot\Big(-\left|
\begin{array}{cc}
\mathbf{f}_{n+1,n}^{0} & \alpha^{t}(n+1) \\
\end{array}
\right|^{2} -\left|
\begin{array}{cc}
\mathbf{f}_{n+1,n}^{0} & \beta^{t}(n+1) \\
\end{array}
\right|^{2}\\
&&+\setlength{\arraycolsep}{0.9pt}
2\cdot\left|
\begin{array}{ccc}
\mathbf{f}_{n+1,n}^{0} & \alpha^{t}(n+1) & \beta^{t}(n+1) \\
\alpha(n) & 0 & 0 \\
\end{array}
\right|\cdot\left|
\begin{array}{ccc}
\mathbf{f}_{n+1,n}^{0} & \alpha^{t}(n+1) & \beta^{t}(n+1) \\
\beta(n) & 0 & 0 \\
\end{array}
\right|\Big) \\
&=&(-1)^{n+1}(-g_{n}^{2}+h_{n}^{2}+2x_{n}y_{n})\equiv g_{n}+h_{n}.
\end{eqnarray*}

%\iffalse

(3)  By the formula (\ref{formula2}),
\begin{eqnarray*}
b_{2n}&=&\left\vert
\begin{array}{ccc}
\mathbf{f}_{2n}^{0} & \alpha ^{t}(2n) & \beta ^{t}(2n) \\
\alpha (2n) & 0 & 0 \\
\beta (2n) & 0 & 0%
\end{array}%
\right\vert =\left\vert
\begin{array}{cccc}
A_{n} & \mathbf{0}_{n,1} & \mathbf{f}_{n}^{0} & \mathbf{1}_{n,1} \\
\mathbf{0}_{1,n} & 0 & \mathbf{1}_{1,n} & 0 \\
\mathbf{f}_{n}^{0} & \mathbf{1}_{n,1} & B_{n} & \mathbf{0}_{n,1} \\
\mathbf{1}_{1,n} & 0 & \mathbf{0}_{1,n} & 0%
\end{array}%
\right\vert  \\
&=&\left\vert
\begin{array}{cccc;{2pt/2pt}cccc}
\mathbf{0}_{n} & \mathbf{0}_{n,1} & \mathbf{0}_{n,1} & \mathbf{0}_{n,1} &
\mathbf{f}_{n}^{0} & \mathbf{1}_{n,1} & -\alpha (n)^{t} & -\beta (n)^{t} \\
\mathbf{0}_{1,n} & 0 & 0 & 0 & \mathbf{1}_{1,n} & 0 & 0 & 0 \\
\mathbf{0}_{1,n} & 0 & 1 & 0 & \beta (n) & 0 & 0 & 0 \\
\mathbf{0}_{1,n} & 0 & 0 & 1 & \alpha (n) & 0 & 0 & 0 \\
\hdotsfor{8}\\
\mathbf{f}_{n}^{0} & \mathbf{1}_{n,1} & -\alpha (n)^{t} & -\beta (n)^{t} &
\mathbf{0}_{n} & \mathbf{0}_{n,1} & \mathbf{0}_{n,1} & \mathbf{0}_{n,1} \\
\mathbf{1}_{1,n} & 0 & 0 & 0 & \mathbf{0}_{1,n} & 0 & 0 & 0 \\
\alpha (n) & 0 & 0 & 0 & \mathbf{0}_{1,n} & 0 & 1 & 0 \\
\beta (n) & 0 & 0 & 0 & \mathbf{0}_{1,n} & 0 & 0 & 1%
\end{array}%
\right\vert  \\
&=&(-1)^{n+1}\cdot \left\vert
\begin{array}{cc}
\mathbf{f}_{n}^{0} & \mathbf{1}_{n,1} \\
\mathbf{1}_{1,n} & 0 \\
\end{array}
\right\vert ^{2} \\
&=&(-1)^{n+1}(c_{n}+2e_{n}+d_{n})^{2}\equiv c_{n}+d_{n}.
\end{eqnarray*}

(4) By the formula (\ref{formula3}),
\begin{eqnarray*}
b_{2n+1}&=&\setlength{\arraycolsep}{0.9pt}
\left\vert
\begin{array}{ccc}
\mathbf{f}_{2n+1}^{0} & \alpha ^{t}(2n+1) & \beta ^{t}(2n+1) \\
\alpha (2n+1) & 0 & 0 \\
\beta (2n+1) & 0 & 0%
\end{array}%
\right\vert =\left\vert
\begin{array}{cccc}
A_{n+1} & \mathbf{0}_{n+1,1} & \mathbf{f}_{n+1,n}^{0} & \mathbf{1}_{n+1,1}
\\
\mathbf{0}_{1,n+1} & 0 & \mathbf{1}_{1,n} & 0 \\
\mathbf{f}_{n,n+1}^{0} & \mathbf{1}_{n,1} & B_{n} & \mathbf{0}_{n,1} \\
\mathbf{1}_{1,n+1} & 0 & \mathbf{0}_{1,n} & 0%
\end{array}%
\right\vert  \\
\end{eqnarray*}
\begin{eqnarray*}
&=&\setlength{\arraycolsep}{0.9pt}
\left\vert
\begin{array}{cccc;{2pt/2pt}cccc}
\mathbf{0}_{n+1} & \mathbf{0}_{n+1,1} & \mathbf{0}_{n+1,1} & \mathbf{0}%
_{n+1,1} & \mathbf{f}_{n+1,n}^{0} & \mathbf{1}_{n+1,1} & -\alpha (n+1)^{t} &
-\beta (n+1)^{t} \\
\mathbf{0}_{1,n+1} & 0 & 0 & 0 & \mathbf{1}_{1,n} & 0 & 0 & 0 \\
\mathbf{0}_{1,n+1} & 0 & 1 & 0 & \beta (n) & 0 & 0 & 0 \\
\mathbf{0}_{1,n+1} & 0 & 0 & 1 & \alpha (n) & 0 & 0 & 0 \\
\hdotsfor{8}\\
\mathbf{f}_{n,n+1}^{0} & \mathbf{1}_{n,1} & -\alpha (n)^{t} & -\beta (n)^{t}
& \mathbf{0}_{n} & \mathbf{0}_{n,1} & \mathbf{0}_{n,1} & \mathbf{0}_{n,1} \\
\mathbf{1}_{1,n+1} & 0 & 0 & 0 & \mathbf{0}_{1,n} & 0 & 0 & 0 \\
\alpha (n+1) & 0 & 0 & 0 & \mathbf{0}_{1,n} & 0 & 1 & 0 \\
\beta (n+1) & 0 & 0 & 0 & \mathbf{0}_{1,n} & 0 & 0 & 1 \\
\end{array}%
\right\vert  \\
&=&(-1)^{n+1}\cdot 2\cdot \left\vert
\begin{array}{ccc}
\mathbf{f}_{n+1,n}^{0} & \alpha ^{t}(n+1) & \beta ^{t}(n+1) \\
\mathbf{1}_{1,n} & 0 & 0 \\
\end{array}
\right\vert ^{2}\\
&=&(-1)^{n+1}2(x_{n}+y_{n})^{2}\equiv 0.
\end{eqnarray*}

(5) By the formula (\ref{formula2}),
\begin{eqnarray*}
c_{2n}&=&\left\vert
\begin{array}{cc}
\mathbf{f}_{2n}^{0} & \alpha ^{t}(2n)  \\
\alpha (2n) & 0  \\
\end{array}
\right\vert =\left\vert
\begin{array}{ccc}
A_{n} & \mathbf{f}_{n}^{0} & \mathbf{1}_{n,1}\\
\mathbf{f}_{n}^{0} & B_{n} & \mathbf{0}_{n,1} \\
\mathbf{1}_{1,n} & \mathbf{0}_{1,n} & 0
\end{array}
\right\vert  \\
&=&\left\vert
\begin{array}{ccc;{2pt/2pt}cccc}
\mathbf{0}_{n} & \mathbf{0}_{n,1} & \mathbf{0}_{n,1}  & \mathbf{f}_{n}^{0} & \mathbf{1}_{n,1} & -\alpha (n)^{t} & -\beta (n)^{t} \\
\mathbf{0}_{1,n} & 1 & 0 & \beta (n) & 0 & 0 & 0 \\
\mathbf{0}_{1,n} & 0 & 1 & \alpha (n) & 0 & 0 & 0 \\
\hdotsfor{7}\\
\mathbf{f}_{n}^{0} & -\alpha (n)^{t} & -\beta (n)^{t} & \mathbf{0}_{n} & \mathbf{0}_{n,1} & \mathbf{0}_{n,1} & \mathbf{0}_{n,1} \\
\mathbf{1}_{1,n} & 0 & 0 & \mathbf{0}_{1,n} & 0 & 0 & 0 \\
\alpha (n) & 0 & 0 & \mathbf{0}_{1,n} & 0 & 1 & 0 \\
\beta (n) & 0 & 0 & \mathbf{0}_{1,n} & 0 & 0 & 1 \\
\end{array}
\right\vert  \\
&=&\setlength{\arraycolsep}{0.9pt}
(-1)^{n}\cdot 2\cdot \left(\left\vert
\begin{array}{ccc}
\mathbf{f}_{n}^{0} & \alpha^{t}(n) & \beta^{t}(n) \\
\alpha(n) & 0 & 0 \\
\beta(n) & 0 & 0 \\
\end{array}
\right\vert ^{2}+\left\vert
\begin{array}{cc}
\mathbf{f}_{n}^{0} & \mathbf{1}_{n,1} \\
\alpha(n) & 0 \\
\end{array}
\right\vert\cdot\left\vert
\begin{array}{cc}
\mathbf{f}_{n}^{0} &  \mathbf{1}_{n,1} \\
\beta(n) & 0 \\
\end{array}
\right\vert\right)\\
&=&(-1)^{n}2\left(b_{n}^{2}+(c_{n}+e_{n})(d_{n}+e_{n})\right)\equiv 0.
\end{eqnarray*}

(6) By the formula (\ref{formula3}),
\begin{eqnarray*}
c_{2n+1}&=&\setlength{\arraycolsep}{0.9pt}
\left\vert
\begin{array}{cc}
\mathbf{f}_{2n+1}^{0} & \alpha^{t}(2n+1)  \\
\alpha(2n+1) & 0
\end{array}%
\right\vert =\left\vert
\begin{array}{ccc}
A_{n+1}& \mathbf{f}_{n+1,n}^{0} & \mathbf{1}_{n+1,1}\\
\mathbf{f}_{n,n+1}^{0}  & B_{n} & \mathbf{0}_{n,1} \\
\mathbf{1}_{1,n+1} & \mathbf{0}_{1,n} & 0
\end{array}
\right\vert  \\
&=&\setlength{\arraycolsep}{0.9pt}
\left\vert
\begin{array}{ccc;{2pt/2pt}cccc}
\mathbf{0}_{n+1} & \mathbf{0}_{n+1,1} & \mathbf{0}_{n+1,1} & \mathbf{f}_{n+1,n}^{0} & \mathbf{1}_{n+1,1} & -\alpha (n+1)^{t} & -\beta (n+1)^{t} \\
\mathbf{0}_{1,n+1} & 1 & 0 & \beta (n) & 0 & 0 & 0 \\
\mathbf{0}_{1,n+1} & 0 & 1 & \alpha (n) & 0 & 0 & 0 \\
\hdotsfor{7}\\
\mathbf{f}_{n,n+1}^{0}  & -\alpha (n)^{t} & -\beta (n)^{t}
& \mathbf{0}_{n} & \mathbf{0}_{n,1} & \mathbf{0}_{n,1}& \mathbf{0}_{n,1} \\
\mathbf{1}_{1,n+1} & 0 & 0 & \mathbf{0}_{1,n} & 0 & 0 & 0 \\
\alpha (n+1) & 0 & 0 & \mathbf{0}_{1,n} & 0 & 1 & 0 \\
\beta (n+1)  & 0 & 0 & \mathbf{0}_{1,n} & 0 & 0 & 1 \\
\end{array}
\right\vert  \\
\end{eqnarray*}
\begin{eqnarray*}
&=&(-1)^{n}\bigg( |
\mathbf{f}_{n+1,n}^{0}  \mathbf{1}_{n+1,1}
|^{2}+ \\
&& \setlength{\arraycolsep}{0.9pt}
4\cdot \left\vert
\begin{array}{ccc}
\mathbf{f}_{n+1,n}^{0} & \alpha ^{t}(n+1) & \beta ^{t}(n+1) \\
\alpha(n) & 0 & 0 \\
\end{array}
\right\vert \cdot \left\vert
\begin{array}{ccc}
\mathbf{f}_{n+1,n}^{0} & \alpha ^{t}(n+1) & \beta ^{t}(n+1) \\
\beta(n) & 0 & 0 \\
\end{array}
\right\vert\bigg)\\
&=&(-1)^{n}\left(4x_{n}y_{n}+(g_{n}+h_{n})^{2}\right)\equiv g_{n}+h_{n}.
\end{eqnarray*}

(7) By the formula (\ref{formula2}),
\begin{eqnarray*}
d_{2n}&=&\left\vert
\begin{array}{cc}
\mathbf{f}_{2n}^{0} & \beta ^{t}(2n)  \\
\beta (2n) & 0  \\
\end{array}
\right\vert=\left\vert
\begin{array}{ccc}
A_{n} & \mathbf{f}_{n}^{0} & \mathbf{0}_{n,1}\\
\mathbf{f}_{n}^{0} & B_{n} & \mathbf{1}_{n,1} \\
\mathbf{0}_{1,n} & \mathbf{1}_{1,n} & 0
\end{array}
\right\vert\\
&=&\left\vert
\begin{array}{cccc;{2pt/2pt}ccc}
\mathbf{0}_{n} & \mathbf{0}_{n,1} & \mathbf{0}_{n,1} & \mathbf{0}_{n,1} & \mathbf{f}_{n}^{0}  & -\alpha (n)^{t} & -\beta (n)^{t} \\
\mathbf{0}_{1,n} & 0& 0 & 0 & \mathbf{1}_{1,n} & 0 & 0  \\
\mathbf{0}_{1,n} & 0& 1 & 0 & \beta (n) & 0 & 0  \\
\mathbf{0}_{1,n} & 0& 0 & 1 & \alpha (n) & 0 & 0  \\
\hdotsfor{7}\\
\mathbf{f}_{n}^{0} & \mathbf{1}_{n,1} &  -\alpha (n)^{t} & -\beta (n)^{t} & \mathbf{0}_{n} & \mathbf{0}_{n,1} & \mathbf{0}_{n,1} \\
\alpha (n) & 0 & 0 & 0 & \mathbf{0}_{1,n}  & 1 & 0 \\
\beta (n) & 0 & 0 & 0 & \mathbf{0}_{1,n} & 0 & 1 \\
\end{array}
\right\vert  \\
&=&\setlength{\arraycolsep}{0.9pt}
(-1)^{n}\cdot \left(2\cdot \left\vert
\begin{array}{ccc}
\mathbf{f}_{n}^{0} & \alpha^{t}(n) & \beta^{t}(n) \\
\alpha(n) & 0 & 0 \\
\beta(n) & 0 & 0 \\
\end{array}
\right\vert ^{2}+\left\vert
\begin{array}{cc}
\mathbf{f}_{n}^{0} & \mathbf{1}_{n,1} \\
\alpha(n) & 0 \\
\end{array}
\right\vert^{2}+\left\vert
\begin{array}{cc}
\mathbf{f}_{n}^{0} &  \mathbf{1}_{n,1} \\
\beta(n) & 0 \\
\end{array}
\right\vert^{2}\right)\\
&=&(-1)^{n}\left(2b_{n}^{2}+(c_{n}+e_{n})^{2}+(d_{n}+e_{n})^{2}\right)\equiv c_{n}+d_{n} .
\end{eqnarray*}

(8) By the formula (\ref{formula3}),
\begin{eqnarray*}
d_{2n+1}&=&\left\vert
\begin{array}{cc}
\mathbf{f}_{2n+1}^{0} & \beta^{t}(2n+1)  \\
\beta(2n+1) & 0
\end{array}%
\right\vert  =\left\vert
\begin{array}{ccc}
A_{n+1}& \mathbf{f}_{n+1,n}^{0} & \mathbf{0}_{n+1,1}\\
\mathbf{f}_{n,n+1}^{0}  & B_{n} & \mathbf{1}_{n,1} \\
\mathbf{0}_{1,n+1} & \mathbf{1}_{1,n} & 0
\end{array}
\right\vert  \\
&=&\setlength{\arraycolsep}{0.9pt}
\left\vert
\begin{array}{cccc;{2pt/2pt}ccc}
\mathbf{0}_{n+1} & \mathbf{0}_{n+1,1} & \mathbf{0}_{n+1,1} & \mathbf{0}_{n+1,1}& \mathbf{f}_{n+1,n}^{0}  & -\alpha (n+1)^{t} & -\beta (n+1)^{t} \\
\mathbf{0}_{1,n+1} & 0 & 0 & 0 & \mathbf{1}_{1,n} & 0 & 0  \\
\mathbf{0}_{1,n+1} & 0 & 1 & 0 & \beta (n) & 0 & 0  \\
\mathbf{0}_{1,n+1} & 0 & 0 & 1 & \alpha (n) & 0 & 0  \\
\hdotsfor{7}\\
\mathbf{f}_{n,n+1}^{0} &\mathbf{1}_{n,1} & -\alpha (n)^{t} & -\beta (n)^{t}
& \mathbf{0}_{n} & \mathbf{0}_{n,1} & \mathbf{0}_{n,1} \\
\alpha (n+1) & 0 & 0 & 0 & \mathbf{0}_{1,n} & 1 & 0 \\
\beta (n+1)  & 0 & 0 & 0 & \mathbf{0}_{1,n} & 0 & 1 \\
\end{array}
\right\vert  \\
&=&(-1)^{n}\cdot \left\vert
\begin{array}{ccc}
\mathbf{f}_{n+1,n}^{0} & \alpha ^{t}(n+1) & \beta ^{t}(n+1) \\
\mathbf{1}_{1,n} & 0 & 0 \\
\end{array}
\right\vert ^{2}\\
&=&(-1)^{n}(x_{n}+y_{n})^{2}\equiv x_{n}+y_{n}.
\end{eqnarray*}

(9) By the formula (\ref{formula2}),
\begin{eqnarray*}
e_{2n}&=&\left\vert
\begin{array}{cc}
\mathbf{f}_{2n}^{0} & \beta ^{t}(2n)  \\
\alpha (2n) & 0  \\
\end{array}
\right\vert =\left\vert
\begin{array}{ccc}
A_{n} & \mathbf{f}_{n}^{0} & \mathbf{0}_{n,1}\\
\mathbf{f}_{n}^{0} & B_{n} & \mathbf{1}_{n,1} \\
\mathbf{1}_{1,n} & \mathbf{0}_{1,n} & 0
\end{array}
\right\vert  \\
\end{eqnarray*}
\begin{eqnarray*}
&=&\left\vert
\begin{array}{ccc;{2pt/2pt}cccc}
\mathbf{0}_{n} & \mathbf{0}_{n,1} & \mathbf{0}_{n,1} & \mathbf{f}_{n}^{0}   & \mathbf{0}_{n,1} & -\alpha (n)^{t} & -\beta (n)^{t} \\
\mathbf{0}_{1,n} & 1 & 0 & \beta (n)  & 0 & 0 & 0  \\
\mathbf{0}_{1,n} & 0 & 1 & \alpha (n) & 0 & 0 & 0  \\
\hdotsfor{7}\\
\mathbf{f}_{n}^{0} &  -\alpha (n)^{t} & -\beta (n)^{t} & \mathbf{0}_{n}  & \mathbf{1}_{n,1} & \mathbf{0}_{n,1} & \mathbf{0}_{n,1} \\
\mathbf{1}_{1,n}  & 0 & 0 & \mathbf{0}_{1,n} & 0 & 0 & 0 \\
\alpha (n) & 0 & 0 & \mathbf{0}_{1,n}  & 0& 1 & 0 \\
\beta (n)  & 0 & 0 & \mathbf{0}_{1,n} & 0& 0 & 1 \\
\end{array}
\right\vert  \\
&=&\setlength{\arraycolsep}{0.9pt}
(-1)^{n}\cdot \Bigg(\bigg(\left\vert
\begin{array}{cc}
\mathbf{f}_{n}^{0} & \alpha^{t}(n) \\
\alpha(n) & 0
\end{array}
\right\vert +\left\vert
\begin{array}{cc}
\mathbf{f}_{n}^{0} & \beta^{t}(n) \\
\beta(n) & 0 \\
\end{array}
\right\vert-2\left\vert
\begin{array}{cc}
\mathbf{f}_{n}^{0} &  \alpha^{t}(n) \\
\beta(n) & 0 \\
\end{array}
\right\vert\bigg)\\
&&\cdot \left\vert
\begin{array}{ccc}
\mathbf{f}_{n}^{0} & \alpha^{t}(n) & \beta^{t}(n) \\
\alpha(n) & 0 & 0 \\
\beta(n) & 0 & 0 \\
\end{array}
\right\vert +|\mathbf{f}_{n}^{0}|\cdot
\left\vert
\begin{array}{cc}
\mathbf{f}_{n}^{0} & \mathbf{1}_{n,1} \\
 \mathbf{1}_{1,n} & 0 \\
 \end{array}
\right\vert\Bigg)\\
&=&(-1)^{n}\Big(b_{n}(c_{n}+d_{n}-2e_{n})+a_{n}(c_{n}+d_{n}+2e_{n})\Big)\\
&\equiv& (a_{n}+b_{n})(c_{n}+d_{n}).
\end{eqnarray*}

(10) By the formula (\ref{formula3}),
\begin{eqnarray*}
e_{2n+1}&=&\left\vert
\begin{array}{cc}
\mathbf{f}_{2n+1}^{0} & \beta^{t}(2n+1)  \\
\alpha(2n+1) & 0
\end{array}%
\right\vert =\left\vert
\begin{array}{ccc}
A_{n+1}& \mathbf{f}_{n+1,n}^{0} & \mathbf{0}_{n+1,1}\\
\mathbf{f}_{n,n+1}^{0}  & B_{n} & \mathbf{1}_{n,1} \\
\mathbf{1}_{1,n+1} & \mathbf{0}_{1,n} & 0
\end{array}
\right\vert  \\
&=&\setlength{\arraycolsep}{0.9pt}
\left\vert
\begin{array}{ccc;{2pt/2pt}cccc}
\mathbf{0}_{n+1} & \mathbf{0}_{n+1,1} & \mathbf{0}_{n+1,1} & \mathbf{f}_{n+1,n}^{0} & \mathbf{0}_{n+1,1} & -\alpha (n+1)^{t} & -\beta (n+1)^{t} \\
\mathbf{0}_{1,n+1}  & 1 & 0 & \beta (n) & 0 & 0 & 0  \\
\mathbf{0}_{1,n+1} & 0 & 1 & \alpha (n) & 0 & 0 & 0  \\
\hdotsfor{7}\\
\mathbf{f}_{n,n+1}^{0}  & -\alpha (n)^{t} & -\beta (n)^{t}
& \mathbf{0}_{n}  & \mathbf{1}_{n,1} & \mathbf{0}_{n,1} & \mathbf{0}_{n,1} \\
\mathbf{1}_{1,n+1} & 0 & 0 & \mathbf{0}_{1,n}  & 0 & 0 & 0 \\
\alpha (n+1) & 0 & 0 & \mathbf{0}_{1,n}  & 0 & 1 & 0 \\
\beta (n+1)   & 0 & 0 & \mathbf{0}_{1,n}  & 0 & 0 & 1 \\
\end{array}
\right\vert  \\
&=&(-1)^{n}\cdot \Big(\left\vert
\begin{array}{cc}
\mathbf{f}_{n+1,n}^{0} & \alpha ^{t}(n+1) \\
\end{array}
\right\vert-\left\vert
\begin{array}{cc}
\mathbf{f}_{n+1,n}^{0} & \beta ^{t}(n+1) \\
\end{array}
\right\vert\Big)\\
&&\cdot \left\vert
\begin{array}{ccc}
\mathbf{f}_{n+1,n}^{0} & \alpha ^{t}(n+1) & \beta^{t}(n+1)\\
\mathbf{1}_{1,n} & 0 & 0
\end{array}
\right\vert\\
&=&(-1)^{n}(g_{n}-h_{n})(x_{n}+y_{n})\equiv (g_{n}+h_{n})(x_{n}+y_{n}).
\end{eqnarray*}

(11) By the formula (\ref{formula2}),
\begin{eqnarray*}
g_{2n}&=&\left\vert
\begin{array}{cc}
\mathbf{f}_{2n+1,2n}^{0} & \alpha ^{t}(2n+1)  \\
\end{array}
\right\vert =(-1)^{n}\cdot\left\vert
\begin{array}{ccc}
A_{n+1,n} & \mathbf{f}_{n+1,n}^{0} & \mathbf{1}_{n+1,1}\\
\mathbf{f}_{n}^{0} & B_{n} & \mathbf{0}_{n,1} \\
\end{array}
\right\vert  \\
&=&\setlength{\arraycolsep}{0.9pt}
(-1)^{n}\cdot\left\vert
\begin{array}{ccc;{2pt/2pt}cccc}
\mathbf{0}_{n+1,n} & \mathbf{0}_{n+1,1} & \mathbf{0}_{n+1,1} & \mathbf{f}_{n+1,n}^{0}   & \mathbf{1}_{n+1,1} & -\alpha (n+1)^{t} & -\beta (n+1)^{t} \\
\mathbf{0}_{1,n} & 1 & 0 & \beta (n)  & 0 & 0 & 0  \\
\mathbf{0}_{1,n} & 0 & 1 & \alpha (n) & 0 & 0 & 0  \\
\hdotsfor{7}\\
\mathbf{f}_{n}^{0} &  -\alpha (n)^{t} & -\beta (n)^{t} & \mathbf{0}_{n}  & \mathbf{0}_{n,1} & \mathbf{0}_{n,1} & \mathbf{0}_{n,1} \\
\alpha (n) & 0 & 0 & \mathbf{0}_{1,n}  & 0& 1 & 0 \\
\beta (n)  & 0 & 0 & \mathbf{0}_{1,n} & 0& 0 & 1 \\
\end{array}
\right\vert  \\
\end{eqnarray*}
\begin{eqnarray*}
&=&(-1)^{n}\cdot \bigg(\left\vert
\begin{array}{ccc}
\mathbf{f}_{n+1,n}^{0} & \alpha^{t}(n+1) &\beta^{t}(n+1)\\
\beta(n) & 0&0
\end{array}
\right\vert\cdot\left\vert
\begin{array}{cc}
\mathbf{f}_{n}^{0} & \alpha^{t}(n) \\
\alpha(n) & 0 \\
\end{array}
\right\vert\\
&&+\left\vert
\begin{array}{ccc}
\mathbf{f}_{n+1,n}^{0} & \alpha^{t}(n+1) &\beta^{t}(n+1)\\
\alpha(n) & 0&0
\end{array}
\right\vert\cdot\left\vert
\begin{array}{cc}
\mathbf{f}_{n}^{0} & \alpha^{t}(n) \\
\beta(n) & 0 \\
\end{array}
\right\vert\\
&&-\left\vert
\begin{array}{ccc}
\mathbf{f}_{n+1,n}^{0} & \alpha^{t}(n+1) &\beta^{t}(n+1)\\
\beta(n) & 0&0
\end{array}
\right\vert\cdot\left\vert
\begin{array}{cc}
\mathbf{f}_{n}^{0} & \alpha^{t}(n) \\
\beta(n) & 0 \\
\end{array}
\right\vert\\
&&\setlength{\arraycolsep}{0.9pt}
-\left\vert
\begin{array}{ccc}
\mathbf{f}_{n+1,n}^{0} & \alpha^{t}(n+1) &\beta^{t}(n+1)\\
\alpha(n) & 0&0
\end{array}
\right\vert\cdot\left\vert
\begin{array}{cc}
\mathbf{f}_{n}^{0} & \beta^{t}(n) \\
\beta(n) & 0 \\
\end{array}
\right\vert\\
&&+
\left\vert
\begin{array}{cc}
\mathbf{f}_{n+1,n}^{0} &  \mathbf{1}_{n+1,1}\\
\end{array}
\right\vert\cdot |\mathbf{f}_{n}|\bigg) \\
&=&(-1)^{n}(c_{n}y_{n}+e_{n}x_{n}-e_{n}y_{n}-d_{n}x_{n}+a_{n}(g_{n}+h_{n}))\\
&\equiv& a_{n}(g_{n}+h_{n})+x_{n}(d_{n}+e_{n})+y_{n}(c_{n}+e_{n}).
\end{eqnarray*}

(12) By the formula (\ref{formula3}),
\begin{eqnarray*}
g_{2n+1}&=&\left\vert
\begin{array}{cc}
\mathbf{f}_{2n+2,2n+1}^{0} & \alpha^{t}(2n+2)  \\
\end{array}%
\right\vert =\left\vert
\begin{array}{ccc}
A_{n+1}& \mathbf{f}_{n+1,n}^{0} & \mathbf{1}_{n+1,1}\\
\mathbf{f}_{n+1}^{0}  & B_{n+1,n} & \mathbf{0}_{n+1,1} \\
\end{array}
\right\vert  \\
&=&\setlength{\arraycolsep}{0.1pt}
\left\vert
\begin{array}{ccc;{2pt/2pt}cccc}
\mathbf{0}_{n+1} & \mathbf{0}_{n+1,1} & \mathbf{0}_{n+1,1} & \mathbf{f}_{n+1,n}^{0} & \mathbf{1}_{n+1,1} & -\alpha (n+1)^{t} & -\beta (n+1)^{t} \\
\mathbf{0}_{1,n+1}  & 1 & 0 & \beta (n) & 0 & 0 & 0  \\
\mathbf{0}_{1,n+1} & 0 & 1 & \alpha (n) & 0 & 0 & 0  \\
\hdotsfor{7}\\
\mathbf{f}_{n+1}^{0}  & -\alpha (n+1)^{t} & -\beta (n+1)^{t}
& \mathbf{0}_{n+1,n}  & \mathbf{0}_{n+1,1} & \mathbf{0}_{n+1,1} & \mathbf{0}_{n+1,1} \\
\alpha (n+1) & 0 & 0 & \mathbf{0}_{1,n}  & 0 & 1 & 0 \\
\beta (n+1)   & 0 & 0 & \mathbf{0}_{1,n}  & 0 & 0 & 1 \\
\end{array}
\right\vert\\
&=&\setlength{\arraycolsep}{0.9pt}
(-1)^{n+1}\cdot\bigg(\left\vert
\begin{array}{ccc}
\mathbf{f}_{n+1,n}^{0} & \alpha^{t}(n+1) &\beta^{t}(n+1)\\
\beta(n) & 0&0
\end{array}
\right\vert\cdot\left\vert
\begin{array}{cc}
\mathbf{f}_{n+1}^{0} & \alpha^{t}(n+1) \\
\alpha(n+1) & 0 \\
\end{array}
\right\vert\\
&&+\left\vert
\begin{array}{ccc}
\mathbf{f}_{n+1,n}^{0} & \alpha^{t}(n+1) &\beta^{t}(n+1)\\
\alpha(n) & 0&0
\end{array}
\right\vert\cdot\left\vert
\begin{array}{cc}
\mathbf{f}_{n+1}^{0} & \alpha^{t}(n+1) \\
\beta(n+1) & 0 \\
\end{array}
\right\vert\\
&&-\left\vert
\begin{array}{ccc}
\mathbf{f}_{n+1,n}^{0} & \alpha^{t}(n+1) &\beta^{t}(n+1)\\
\beta(n) & 0&0
\end{array}
\right\vert\cdot\left\vert
\begin{array}{cc}
\mathbf{f}_{n+1}^{0} & \alpha^{t}(n+1) \\
\beta(n+1) & 0 \\
\end{array}
\right\vert\\
&&\setlength{\arraycolsep}{0.01pt}
-\left\vert
\begin{array}{ccc}
\mathbf{f}_{n+1,n}^{0} & \alpha^{t}(n+1) &\beta^{t}(n+1)\\
\alpha(n) & 0&0
\end{array}
\right\vert\cdot
\left\vert
\begin{array}{cc}
\mathbf{f}_{n+1}^{0} & \beta^{t}(n+1) \\
\beta(n+1) & 0 \\
\end{array}
\right\vert\\
&&+|
\mathbf{f}_{n+1,n}^{0}  \mathbf{1}_{n+1,1}
|\cdot |\mathbf{f}_{n+1}|\bigg)\\
&=&(-1)^{n+1}\Big(c_{n+1}y_{n}+e_{n+1}x_{n}-e_{n+1}y_{n}-d_{n+1}x_{n}+a_{n+1}(g_{n}+h_{n})\Big)\\
&\equiv& a_{n+1}(g_{n}+h_{n})+x_{n}(d_{n+1}+e_{n+1})+y_{n}(c_{n+1}+e_{n+1}).
\end{eqnarray*}

(13) By the formula (\ref{formula2}),
\begin{eqnarray*}
h_{2n}&=&\left\vert
\begin{array}{cc}
\mathbf{f}_{2n+1,2n}^{0} & \beta ^{t}(2n+1)  \\
\end{array}
\right\vert =\left\vert
\begin{array}{ccc}
A_{n+1,n} & \mathbf{0}_{n+1,1}& \mathbf{f}_{n+1,n}^{0} \\
\mathbf{f}_{n}^{0} & \mathbf{1}_{n,1} & B_{n}  \\
\end{array}
\right\vert  \\
\end{eqnarray*}
\begin{eqnarray*}
&=&\setlength{\arraycolsep}{0.9pt}
(-1)^{n}\cdot\left\vert
\begin{array}{cccc;{2pt/2pt}ccc}
\mathbf{0}_{n+1,n} & \mathbf{0}_{n+1,1} & \mathbf{0}_{n+1,1}   & \mathbf{0}_{n+1,1} & \mathbf{f}_{n+1,n}^{0}  & -\alpha (n+1)^{t} & -\beta (n+1)^{t} \\
\mathbf{0}_{1,n} & 0  & 1 & 0 & \beta (n)  & 0 & 0  \\
\mathbf{0}_{1,n} & 0  & 0 & 1 & \alpha (n) & 0 & 0  \\
\hdotsfor{7}\\
\mathbf{f}_{n}^{0}  & \mathbf{1}_{n,1} &  -\alpha (n)^{t} & -\beta (n)^{t} & \mathbf{0}_{n}  & \mathbf{0}_{n,1} & \mathbf{0}_{n,1} \\
\alpha (n) & 0& 0 & 0 & \mathbf{0}_{1,n}  & 1 & 0 \\
\beta (n)  & 0& 0 & 0 & \mathbf{0}_{1,n} & 0 & 1 \\
\end{array}
\right\vert  \\
&=&\setlength{\arraycolsep}{0.9pt}
(-1)^{n}\cdot \Bigg(\bigg(\left\vert
\begin{array}{ccc}
\mathbf{f}_{n+1,n}^{0} & \alpha^{t}(n+1) &\beta^{t}(n+1)\\
\beta(n) & 0&0
\end{array}
\right\vert-\left\vert
\begin{array}{ccc}
\mathbf{f}_{n+1,n}^{0}& \alpha^{t}(n+1) &\beta^{t}(n+1)\\
\alpha(n) & 0&0
\end{array}
\right\vert\bigg)\\
&&\cdot\left\vert
\begin{array}{ccc}
\mathbf{f}_{n}^{0} & \alpha^{t}(n) &\beta^{t}(n)\\
\alpha(n) & 0&0\\
\beta(n) & 0&0\\
\end{array}
\right\vert+|
\mathbf{f}_{n+1,n}^{0}  \alpha^{t}(n+1) |\cdot\left\vert
\begin{array}{cc}
\mathbf{f}_{n}^{0} & \alpha^{t}(n)\\
\mathbf{1}_{1,n} & 0
\end{array}
\right\vert\\
&&+\left\vert
\begin{array}{cc}
\mathbf{f}_{n+1,n}^{0} & \beta^{t}(n+1) \\
\end{array}
\right\vert\cdot\left\vert
\begin{array}{cc}
\mathbf{f}_{n}^{0} & \beta^{t}(n)\\
\mathbf{1}_{1,n} & 0
\end{array}
\right\vert\Bigg)\\
&=&(-1)^{n}\Big(b_{n}(y_{n}-x_{n})+g_{n}(c_{n}+e_{n})+h_{n}(d_{n}+e_{n})\Big)\\
&\equiv& g_{n}(c_{n}+e_{n})+h_{n}(d_{n}+e_{n})+b_{n}(x_{n}+y_{n}).
\end{eqnarray*}

(14) By the formula (\ref{formula3}),
\begin{eqnarray*}
h_{2n+1}&=&\left\vert
\begin{array}{cc}
\mathbf{f}_{2n+2,2n+1}^{0} & \beta^{t}(2n+2)  \\
\end{array}%
\right\vert =\left\vert
\begin{array}{ccc}
A_{n+1}& \mathbf{f}_{n+1,n}^{0} & \mathbf{0}_{n+1,1}\\
\mathbf{f}_{n+1}^{0}  & B_{n+1,n} & \mathbf{1}_{n+1,1} \\
\end{array}
\right\vert  \\
&=&\setlength{\arraycolsep}{0.9pt}
\left\vert
\begin{array}{ccc;{2pt/2pt}cccc}
\mathbf{0}_{n+1} & \mathbf{0}_{n+1,1} & \mathbf{0}_{n+1,1} & \mathbf{f}_{n+1,n}^{0} & \mathbf{0}_{n+1,1} & -\alpha (n+1)^{t} & -\beta (n+1)^{t} \\
\mathbf{0}_{1,n+1}  & 1 & 0 & \beta (n) & 0 & 0 & 0  \\
\mathbf{0}_{1,n+1} & 0 & 1 & \alpha (n) & 0 & 0 & 0  \\
\hdotsfor{7}\\
\mathbf{f}_{n+1}^{0}  & -\alpha (n+1)^{t} & -\beta (n+1)^{t}
& \mathbf{0}_{n+1,n}  & \mathbf{1}_{n+1,1} & \mathbf{0}_{n+1,1} & \mathbf{0}_{n+1,1} \\
\alpha (n+1) & 0 & 0 & \mathbf{0}_{1,n}  & 0 & 1 & 0 \\
\beta (n+1)   & 0 & 0 & \mathbf{0}_{1,n}  & 0 & 0 & 1 \\
\end{array}
\right\vert\\
&=&\setlength{\arraycolsep}{0.1pt}
(-1)^{n+1}\cdot \Bigg(\bigg(\left\vert
\begin{array}{ccc}
\mathbf{f}_{n+1,n}^{0} & \alpha^{t}(n+1) &\beta^{t}(n+1)\\
\beta(n) & 0&0
\end{array}
\right\vert-\left\vert
\begin{array}{ccc}
\mathbf{f}_{n+1,n}^{0}& \alpha^{t}(n+1) &\beta^{t}(n+1)\\
\alpha(n) & 0&0
\end{array}
\right\vert\bigg)\\
&&\setlength{\arraycolsep}{0.9pt}
\cdot\left\vert
\begin{array}{ccc}
\mathbf{f}_{n+1}^{0} & \alpha^{t}(n+1) &\beta^{t}(n+1)\\
\alpha(n+1) & 0&0\\
\beta(n+1) & 0&0\\
\end{array}
\right\vert+|
\mathbf{f}_{n+1,n}^{0}  \alpha^{t}(n+1) |\cdot\left\vert
\begin{array}{cc}
\mathbf{f}_{n+1}^{0} & \alpha^{t}(n+1)\\
\mathbf{1}_{1,n+1} & 0
\end{array}
\right\vert\\
&&+\left\vert
\begin{array}{cc}
\mathbf{f}_{n+1,n}^{0} & \beta^{t}(n+1) \\
\end{array}
\right\vert\cdot\left\vert
\begin{array}{cc}
\mathbf{f}_{n+1}^{0} & \beta^{t}(n+1)\\
\mathbf{1}_{1,n+1} & 0
\end{array}
\right\vert\Bigg)\\
&=&(-1)^{n+1}\Big(b_{n+1}(y_{n}-x_{n})+g_{n}(c_{n+1}+e_{n+1})+h_{n}(d_{n+1}+e_{n+1})\Big)\\
&\equiv& g_{n}(c_{n+1}+e_{n+1})+h_{n}(d_{n+1}+e_{n+1})+b_{n+1}(x_{n}+y_{n}).
\end{eqnarray*}

(15) By the formula (\ref{formula2}),
\begin{eqnarray*}
x_{2n}&=&\setlength{\arraycolsep}{0.1pt}
\left\vert
\begin{array}{ccc}
\mathbf{f}_{2n+1,2n}^{0} & \alpha ^{t}(2n+1)& \beta ^{t}(2n+1) \\
\alpha(2n) &0 & 0
\end{array}
\right\vert
=- \left\vert
\begin{array}{cc;{2pt/2pt}cc}
A_{n+1,n} & \mathbf{0}_{n+1,1}& \mathbf{f}_{n+1,n}^{0} &\mathbf{1}_{n+1,1} \\
\hdotsfor{4}\\
\mathbf{f}_{n}^{0} & \mathbf{1}_{n,1} & B_{n} & \mathbf{0}_{n,1} \\
\mathbf{1}_{1,n}&0&\mathbf{0}_{1,n}&0
\end{array}
\right\vert  \\
\end{eqnarray*}
\begin{eqnarray*}
&=&\setlength{\arraycolsep}{0.9pt}
-\left\vert
\begin{array}{cccc;{2pt/2pt}cccc}
\mathbf{0}_{n+1,n} & \mathbf{0}_{n+1,1} & \mathbf{0}_{n+1,1}   & \mathbf{0}_{n+1,1} & \mathbf{f}_{n+1,n}^{0}  & \mathbf{1}_{n+1,1} & -\alpha (n+1)^{t} & -\beta (n+1)^{t} \\
\mathbf{0}_{1,n} & 0  & 1 & 0 & \beta (n) & 0 & 0 & 0  \\
\mathbf{0}_{1,n} & 0  & 0 & 1 & \alpha (n) & 0 & 0 & 0  \\
\hdotsfor{8}\\
\mathbf{f}_{n}^{0}  & \mathbf{1}_{n,1} &  -\alpha (n)^{t} & -\beta (n)^{t} & \mathbf{0}_{n}  & \mathbf{0}_{n,1} & \mathbf{0}_{n,1} & \mathbf{0}_{n,1}\\
\mathbf{1}_{1,n} &0 &0 &0 & \mathbf{0}_{1,n} &0 &0 &0\\
\alpha (n) & 0& 0 & 0 & \mathbf{0}_{1,n}  &0  & 1 & 0 \\
\beta (n)  & 0& 0 & 0 & \mathbf{0}_{1,n} &0 & 0 & 1 \\
\end{array}
\right\vert  \\
&=&\setlength{\arraycolsep}{0.1pt}
(-1)^{n}\cdot \Bigg(
\bigg(\left\vert
\begin{array}{ccc}
\mathbf{f}_{n+1,n}^{0}& \alpha^{t}(n+1) &\beta^{t}(n+1)\\
\beta(n) & 0&0
\end{array}
\right\vert-\left\vert
\begin{array}{ccc}
\mathbf{f}_{n+1,n}^{0}& \alpha^{t}(n+1) &\beta^{t}(n+1)\\
\alpha(n) & 0&0
\end{array}
\right\vert\bigg)\\
&&\cdot2\left\vert
\begin{array}{ccc}
\mathbf{f}_{n}^{0} & \alpha^{t}(n) &\beta^{t}(n)\\
\alpha(n) & 0 & 0\\
\beta(n) & 0  & 0\\
\end{array}
\right\vert+\left\vert
\begin{array}{cc}
\mathbf{f}_{n+1,n}^{0} & \mathbf{1}_{n+1,1}\\
\end{array}
\right\vert\cdot\left\vert
\begin{array}{cc}
\mathbf{f}_{n}^{0} & \mathbf{1}_{n,1}\\
\mathbf{1}_{1,n} &0
\end{array}
\right\vert\Bigg)\\
&=&(-1)^{n}\Big(2b_{n}(y_{n}-x_{n})+(g_{n}+h_{n})(c_{n}+d_{n}+2e_{n})\Big)\\
&\equiv &
(g_{n}+h_{n})(c_{n}+d_{n}).
\end{eqnarray*}

(16) By the formula (\ref{formula3}),
\begin{eqnarray*}
x_{2n+1}&=&
\setlength{\arraycolsep}{0.9pt}
\begin{vmatrix}
\mathbf{f}_{2n+2,2n+1}^{0} & \alpha^{t}(2n+2) & \beta^{t}(2n+2)  \\
 \alpha^{t}(2n+1) & 0 & 0
\end{vmatrix} \\
&=&
 (-1)^{n+1}\cdot\left\vert
\begin{array}{cc;{2pt/2pt}cc}
A_{n+1}  & \mathbf{0}_{n+1,1} & \mathbf{f}_{n+1,n}^{0} & \mathbf{1}_{n+1,1}\\
\hdotsfor{4}\\
\mathbf{f}_{n+1}^{0}  & \mathbf{1}_{n+1,1} & B_{n+1,n} & \mathbf{0}_{n+1,1} \\
\mathbf{1}_{1,n+1} & 0 & \mathbf{0}_{1,n}& 0
\end{array}
\right\vert\\
&=&\setlength{\arraycolsep}{0.2pt}
\scriptsize(-1)^{n+1}\left\vert
\begin{array}{cccc;{2pt/2pt}cccc}
\mathbf{0}_{n+1} & \mathbf{0}_{n+1,1} & \mathbf{0}_{n+1,1}  & \mathbf{0}_{n+1,1} & \mathbf{f}_{n+1,n}^{0} & \mathbf{1}_{n+1,1} & -\alpha (n+1)^{t} & -\beta (n+1)^{t} \\
\mathbf{0}_{1,n+1}  & 0 & 1 & 0 & \beta (n) & 0 & 0 & 0  \\
\mathbf{0}_{1,n+1} & 0 & 0 & 1 & \alpha (n) & 0 & 0 & 0  \\
\hdotsfor{8}\\
\mathbf{f}_{n+1}^{0} & \mathbf{1}_{n+1,1} & -\alpha (n+1)^{t} & -\beta (n+1)^{t}
& \mathbf{0}_{n+1,n}  & \mathbf{0}_{n+1,1} & \mathbf{0}_{n+1,1} & \mathbf{0}_{n+1,1} \\
\mathbf{1}_{1,n+1} & 0 & 0 & 0 & \mathbf{0}_{1,n}  & 0 & 0 & 0 \\
\alpha (n+1) & 0& 0 & 0 & \mathbf{0}_{1,n}  & 0 & 1 & 0 \\
\beta (n+1) & 0  & 0 & 0 & \mathbf{0}_{1,n}  & 0 & 0 & 1 \\
\end{array}
\right\vert\\
&=&\setlength{\arraycolsep}{0.01pt}
(-1)^{n+1}\cdot \bigg(\scriptsize
|\mathbf{f}_{n+1,n}^{0} \mathbf{1}_{n+1,1}|\cdot\begin{vmatrix}
\mathbf{f}_{n+1}^{0} & \mathbf{1}_{n+1,1}\\
\mathbf{1}_{1,n+1} &0
\end{vmatrix}+2\begin{vmatrix}
\mathbf{f}_{n+1}^{0} & \alpha^{t}(n+1) &\beta^{t}(n+1)\\
\alpha(n+1) & 0 & 0\\
\beta(n+1) & 0  & 0\\
\end{vmatrix}\\
&&\cdot\Big(\begin{vmatrix}
\mathbf{f}_{n+1,n}^{0}& \alpha^{t}(n+1) &\beta^{t}(n+1)\\
\beta(n) & 0&0
\end{vmatrix}-\begin{vmatrix}
\mathbf{f}_{n+1,n}^{0}& \alpha^{t}(n+1) &\beta^{t}(n+1)\\
\alpha(n) & 0&0
\end{vmatrix}\Big)\Bigg)\\
&=&(-1)^{n+1}\Big(2b_{n+1}(y_{n}-x_{n})+(g_{n}+h_{n})(c_{n+1}+d_{n+1}+2e_{n+1})\Big)\\
&\equiv & (g_{n}+h_{n})(c_{n+1}+d_{n+1}).
\end{eqnarray*}

(17) By the formula (\ref{formula2}),
\begin{eqnarray*}
y_{2n}&=&\setlength{\arraycolsep}{0.2pt}
\begin{vmatrix}
\mathbf{f}_{2n+1,2n}^{0} & \alpha ^{t}(2n+1)& \beta ^{t}(2n+1) \\
\beta(2n) &0 & 0
\end{vmatrix}=-\begin{vmatrix}
A_{n+1,n} & \mathbf{0}_{n+1,1}& \mathbf{f}_{n+1,n}^{0} &\mathbf{1}_{n+1,1} \\
\mathbf{f}_{n}^{0} & \mathbf{1}_{n,1} & B_{n} & \mathbf{0}_{n,1} \\
\mathbf{0}_{1,n}&0&\mathbf{1}_{1,n}&0
\end{vmatrix}  \\
\end{eqnarray*}
\begin{eqnarray*}
&=&\setlength{\arraycolsep}{0.2pt}
(-1)^{n+1}\left\vert
\begin{array}{cccc;{2pt/2pt}cccc}
\mathbf{0}_{n+1,n} & \mathbf{0}_{n+1,1} & \mathbf{0}_{n+1,1}   & \mathbf{0}_{n+1,1} & \mathbf{f}_{n+1,n}^{0}  & \mathbf{1}_{n+1,1} & -\alpha (n+1)^{t} & -\beta (n+1)^{t} \\
\mathbf{0}_{1,n} & 0  & 0 & 0 & \mathbf{1}_{1,n} & 0  & 0 & 0\\
\mathbf{0}_{1,n} & 0  & 1 & 0 & \beta (n) & 0 & 0 & 0  \\
\mathbf{0}_{1,n} & 0  & 0 & 1 & \alpha (n) & 0 & 0 & 0  \\
\hdotsfor{8}\\
\mathbf{f}_{n}^{0}  & \mathbf{1}_{n,1} &  -\alpha (n)^{t} & -\beta (n)^{t} & \mathbf{0}_{n}  & \mathbf{0}_{n,1} & \mathbf{0}_{n,1} & \mathbf{0}_{n,1}\\
\alpha (n) & 0& 0 & 0 & \mathbf{0}_{1,n}  &0  & 1 & 0 \\
\beta (n)  & 0& 0 & 0 & \mathbf{0}_{1,n} &0 & 0 & 1 \\
\end{array}
\right\vert \\
&=&\setlength{\arraycolsep}{0.9pt}
(-1)^{n+1}\left\vert
\begin{array}{ccc}
\mathbf{f}_{n+1,n}^{0} & \alpha^{t}(n+1) &\beta^{t}(n+1)\\
\mathbf{1}_{1,n} & 0 & 0\\
\end{array}
\right\vert
\cdot\bigg(\left\vert
\begin{array}{ccc}
\mathbf{f}_{n}^{0}& \mathbf{1}_{n,1}\\
\alpha(n) & 0
\end{array}
\right\vert-\left\vert
\begin{array}{ccc}
\mathbf{f}_{n}^{0}& \mathbf{1}_{n,1}\\
\beta(n) & 0
\end{array}
\right\vert\bigg)\\
&=&(-1)^{n+1}(x_{n}+y_{n})(c_{n}-d_{n})\equiv
(x_{n}+y_{n})(c_{n}+d_{n}).
\end{eqnarray*}

(18) By the formula (\ref{formula3}),
\begin{eqnarray*}
y_{2n+1}&=&\setlength{\arraycolsep}{0.2pt}
\begin{vmatrix}
\mathbf{f}_{2n+2,2n+1}^{0} & \alpha^{t}(2n+2) & \beta^{t}(2n+2)  \\
 \beta^{t}(2n+1) & 0 & 0
\end{vmatrix} =
\begin{vmatrix}
A_{n+1}& \mathbf{f}_{n+1,n}^{0} & \mathbf{1}_{n+1,1}& \mathbf{0}_{n+1,1}\\
\mathbf{f}_{n+1}^{0}  & B_{n+1,n} & \mathbf{0}_{n+1,1} & \mathbf{1}_{n+1,1}\\
\mathbf{0}_{1,n+1} & \mathbf{1}_{1,n}& 0 & 0
\end{vmatrix}\\
&=&\setlength{\arraycolsep}{0.2pt}
\left\vert
\begin{array}{cccc;{2pt/2pt}cccc}
 \mathbf{0}_{n+1} & \mathbf{0}_{n+1,1} & \mathbf{0}_{n+1,1}  & \mathbf{0}_{n+1,1} & \mathbf{f}_{n+1,n}^{0} & \mathbf{1}_{n+1,1} & -\alpha (n+1)^{t} & -\beta (n+1)^{t} \\
\mathbf{0}_{1,n+1} & 0 & 0 & 0 & \mathbf{1}_{1,n}  & 0 & 0 & 0 \\
\mathbf{0}_{1,n+1}  & 0 & 1 & 0 & \beta (n) & 0 & 0 & 0  \\
\mathbf{0}_{1,n+1} & 0 & 0 & 1 & \alpha (n) & 0 & 0 & 0  \\
\hdotsfor{8}\\
\mathbf{f}_{n+1}^{0} & \mathbf{1}_{n+1,1} & -\alpha (n+1)^{t} & -\beta (n+1)^{t}
& \mathbf{0}_{n+1,n}  & \mathbf{0}_{n+1,1} & \mathbf{0}_{n+1,1} & \mathbf{0}_{n+1,1} \\
\alpha (n+1) & 0& 0 & 0 & \mathbf{0}_{1,n}  & 0 & 1 & 0 \\
\beta (n+1) & 0  & 0 & 0 & \mathbf{0}_{1,n}  & 0 & 0 & 1 \\
 \end{array}
\right\vert\\
&=&(-1)^{n}\cdot \begin{vmatrix}
    \begin{smallmatrix}
\mathbf{f}_{n+1,n}^{0} & \alpha^{t}(n+1) &\beta^{t}(n+1)\\
\mathbf{1}_{1,n} & 0 & 0\\
\end{smallmatrix}
\end{vmatrix}
\cdot\bigg(\begin{vmatrix}
    \begin{smallmatrix}
\mathbf{f}_{n+1}^{0}& \mathbf{1}_{n+1,1}\\
\alpha(n+1) & 0
\end{smallmatrix}
\end{vmatrix}-\begin{vmatrix}
    \begin{smallmatrix}
\mathbf{f}_{n+1}^{0}& \mathbf{1}_{n+1,1}\\
\beta(n+1) & 0
\end{smallmatrix}
\end{vmatrix}\bigg)\\
&=&(-1)^{n}(x_{n}+y_{n})(c_{n+1}-d_{n+1})\equiv
(x_{n}+y_{n})(c_{n+1}+d_{n+1}).
\end{eqnarray*}

%\fi

\end{proof}
\begin{remark}
In this paper, we only consider the Hankel determinant $|\mathbf{f}_{n}^{p}|$
with the case $p=0$. In fact, the proof of the Lemma can be adapted to the
other cases with $p\geq1$.
\end{remark}

By Lemma \ref{lemma1}, we have following proposition.

\begin{proposition}
\label{lemma2} For any integer $n\geq1$, we have

\begin{enumerate}
\item $a_{n}\equiv\left\{
\begin{array}{ll}
1, & \hbox{if $n\equiv 0,1,2,5,8,9~ (\bmod 10)$;} \\
0, & \hbox{otherwise.}%
\end{array}
\right.$

\item $b_{n}\equiv \left\{
\begin{array}{ll}
1, & \hbox{if $n\equiv 2,4,6,8~(\bmod 10)$;} \\
0, & \hbox{otherwise.}%
\end{array}%
\right. $

\item $c_{n}\equiv\left\{
\begin{array}{ll}
1, & \hbox{if $n\equiv 1,5,9~(\bmod 10)$;} \\
0, & \hbox{otherwise.}%
\end{array}
\right.$

\item $d_{n}\equiv\left\{
\begin{array}{ll}
1, & \hbox{if $n\equiv 2,3,4,5,6,7,8~(\bmod 10)$;} \\
0, & \hbox{otherwise.}%
\end{array}
\right.$

\item $e_{n}\equiv\left\{
\begin{array}{ll}
1, & \hbox{if $n\equiv 2,5,8~(\bmod 10)$;} \\
0, & \hbox{otherwise.}%
\end{array}
\right.$

\item $g_{n}\equiv\left\{
\begin{array}{ll}
1, & \hbox{if $n\equiv 0,1,8,9~(\bmod 10)$;} \\
0, & \hbox{otherwise.}%
\end{array}
\right.$

\item $h_{n}\equiv\left\{
\begin{array}{ll}
1, & \hbox{if $n\equiv 1,2,4,5,7,8~(\bmod 10)$;} \\
0, & \hbox{otherwise.}%
\end{array}
\right.$

\item $x_{n}\equiv\left\{
\begin{array}{ll}
1, & \hbox{if $n\equiv 1,4,5,8~(\bmod 10)$;} \\
0, & \hbox{otherwise.}%
\end{array}
\right.$

\item $y_{n}\equiv\left\{
\begin{array}{ll}
1, & \hbox{if $n\equiv 2,3,4,5,6,7~(\bmod 10)$;} \\
0, & \hbox{otherwise.}%
\end{array}
\right.$
\end{enumerate}
\end{proposition}
\begin{remark}
The first assertion gives a positive answer of Coons's conjecture \cite{CV12} that the sequence $\{H_{n}(F)\}_{n\geq0}$ (module 2) is periodic with period 10.
\end{remark}

\begin{proof}
We easily check that the conclusions above are true for $n\leq10$. Now, assume the
conclusions are true for $n\leq 10k$ with $k\geq1$, then we prove the case $%
10k+1\leq n\leq 10(k+1)$ by induction. We have following nine cases to discuss.
\begin{enumerate}
  \item By the equalities (1) and (2) of Lemma \ref{lemma1}, we have following ten subcases.
           \begin{itemize}
           \item $a_{10k+1}\equiv g_{5k}+h_{5k}\equiv \left\{
           \begin{array}{ll}
           g_{10l}+h_{10l}, & \hbox{if $k=2l$;} \\
           g_{10l+5}+h_{10l+5}, & \hbox{if $k=2l+1$.}%
           \end{array}
           \right.$

           Since $10l,10l+5\leq 10k$, then, by the hypothesis, $g_{10l}\equiv
           h_{10l+5}\equiv 1,h_{10l}\equiv g_{10l+5}\equiv 0$. Hence, $a_{10k+1}\equiv
           1.$

           \item $a_{10k+2}\equiv a_{5k+1}+b_{5k+1}\equiv \left\{
           \begin{array}{ll}
           a_{10l+1}+b_{10l+1}, & \hbox{if $k=2l$;} \\
            a_{10l+6}+b_{10l+6}, & \hbox{if $k=2l+1$.}%
           \end{array}
           \right.$

           Since $10l+1,10l+6\leq 10k$, then, by the hypothesis, $a_{10l+1}\equiv
           b_{10l+6}\equiv 1,a_{10l+6}\equiv b_{10l+1}\equiv 0$. Hence, $%
           a_{10k+2}\equiv 1.$

           \item $a_{10k+3}\equiv g_{5k+1}+h_{5k+1}\equiv \left\{
            \begin{array}{ll}
           g_{10l+1}+h_{10l+1}, & \hbox{if $k=2l$;} \\
           g_{10l+6}+h_{10l+6}, & \hbox{if $k=2l+1$.}%
           \end{array}
           \right.$

           Since $10l+1,10l+6\leq 10k$, then, by the hypothesis, $g_{10l+1}\equiv
           h_{10l+1}\equiv 1,g_{10l+6}\equiv h_{10l+6}\equiv 0$. Hence, $%
           a_{10k+3}\equiv 0.$

           \item $a_{10k+4}\equiv a_{5k+2}+b_{5k+2}\equiv \left\{
           \begin{array}{ll}
           a_{10l+2}+b_{10l+2}, & \hbox{if $k=2l$;} \\
            a_{10l+7}+b_{10l+7}, & \hbox{if $k=2l+1$.}%
           \end{array}
           \right.$

           Since $10l+2,10l+7\leq 10k$, then, by the hypothesis, $a_{10l+2}\equiv
           b_{10l+2}\equiv 1,a_{10l+7}\equiv b_{10l+7}\equiv 0$. Hence, $%
           a_{10k+4}\equiv 0.$

           \item $a_{10k+5}\equiv g_{5k+2}+h_{5k+2}\equiv \left\{
           \begin{array}{ll}
           g_{10l+2}+h_{10l+2}, & \hbox{if $k=2l$;} \\
           g_{10l+7}+h_{10l+7}, & \hbox{if $k=2l+1$.}%
           \end{array}
           \right.$

           Since $10l+2,10l+7\leq 10k$, then, by the hypothesis, $g_{10l+2}\equiv
           h_{10l+7}\equiv 0,g_{10l+7}\equiv g_{10l+2}\equiv 1$. Hence, $%
           a_{10k+5}\equiv 1.$

           \item $a_{10k+6}\equiv a_{5k+3}+b_{5k+3}\equiv \left\{
           \begin{array}{ll}
           a_{10l+3}+b_{10l+3}, & \hbox{if $k=2l$;} \\
           a_{10l+8}+b_{10l+8}, & \hbox{if $k=2l+1$.}%
           \end{array}
           \right.$

           Since $10l+3,10l+8\leq 10k$, then, by the hypothesis, $a_{10l+3}\equiv
           b_{10l+3}\equiv 0,a_{10l+8}\equiv b_{10l+8}\equiv 1$. Hence, $%
           a_{10k+6}\equiv 0.$

           \item $a_{10k+7}\equiv g_{5k+3}+h_{5k+3}\equiv \left\{
           \begin{array}{ll}
           g_{10l+3}+h_{10l+3}, & \hbox{if $k=2l$;} \\
           g_{10l+8}+h_{10l+8}, & \hbox{if $k=2l+1$.}%
           \end{array}
           \right.$

           Since $10l+3,10l+8\leq 10k$, then, by the hypothesis, $g_{10l+3}\equiv
           h_{10l+3}\equiv 0,g_{10l+6}\equiv h_{10l+1}\equiv 1$. Hence, $%
           a_{10k+7}\equiv 0.$

           \item $a_{10k+8}\equiv a_{5k+4}+b_{5k+4}\equiv \left\{
           \begin{array}{ll}
           a_{10l+4}+b_{10l+4}, & \hbox{if $k=2l$;} \\
           a_{10l+9}+b_{10l+9}, & \hbox{if $k=2l+1$.}%
           \end{array}
           \right.$

           Since $10l+4,10l+9\leq 10k$, then, by the hypothesis, $a_{10l+4}\equiv
           b_{10l+9}\equiv 0,a_{10l+9}\equiv b_{10l+4}\equiv 1$. Hence, $%
            a_{10k+8}\equiv 1.$

           \item $a_{10k+9}\equiv g_{5k+4}+h_{5k+4}\equiv \left\{
           \begin{array}{ll}
           g_{10l+4}+h_{10l+4}, & \hbox{if $k=2l$;} \\
           g_{10l+9}+h_{10l+9}, & \hbox{if $k=2l+1$.}%
           \end{array}
           \right.$

           Since $10l+4,10l+9\leq 10k$, then, by the hypothesis, $g_{10l+4}\equiv
           h_{10l+9}\equiv 0,g_{10l+9}\equiv h_{10l+4}\equiv 1$. Hence, $%
           a_{10k+9}\equiv 1.$

           \item $a_{10k+10}\equiv a_{5k+5}+b_{5k+5}\equiv \left\{
           \begin{array}{ll}
           a_{10l+5}+b_{10l+5}, & \hbox{if $k=2l$;} \\
           a_{10l+10}+b_{10l+10}, & \hbox{if $k=2l+1$.}%
           \end{array}
           \right.$

           Since $10l+5,10l+10\leq 10k$, then, by the hypothesis, $a_{10l+5}\equiv
           a_{10l+10}\equiv 1,b_{10l+5}\equiv b_{10l+10}\equiv 0$. Hence, $%
           a_{10k+10}\equiv 1.$
           \end{itemize}

          Thus, the first assertion is true for $n\leq 10(k+1)$, which completes the proof  the first assertion.
  \item By the equalities (3) and (4) of Lemma \ref{lemma1}, we have following six subcases.
           \begin{itemize}
           \item $b_{10k+i}\equiv 0$ for $i=1,3,5,7,9$, since $b_{2n+1}\equiv0$ for $n\geq1$.
           \item $b_{10k+2}\equiv c_{5k+1}+d_{5k+1}\equiv \left\{
           \begin{array}{ll}
           c_{10l+1}+d_{10l+1}, & \hbox{if $k=2l$;} \\
           c_{10l+6}+d_{10l+6}, & \hbox{if $k=2l+1$.}%
           \end{array}
           \right.$

           Since $10l+1,10l+6\leq 10k$, then, by the hypothesis, $c_{10l+1}\equiv
           d_{10l+6}\equiv 1,c_{10l+6}\equiv d_{10l+1}\equiv 0$. Hence, $%
           b_{10k+2}\equiv 1.$

           \item $b_{10k+4}\equiv c_{5k+2}+d_{5k+2}\equiv \left\{
           \begin{array}{ll}
           c_{10l+2}+d_{10l+2}, & \hbox{if $k=2l$;} \\
           c_{10l+7}+d_{10l+7}, & \hbox{if $k=2l+1$.}%
           \end{array}
           \right.$

           Since $10l+2,10l+7\leq 10k$, then, by the hypothesis, $c_{10l+2}\equiv
           c_{10l+7}\equiv 0,d_{10l+2}\equiv d_{10l+7}\equiv 0$. Hence, $%
           b_{10k+4}\equiv 1.$

           \item $b_{10k+6}\equiv c_{5k+3}+d_{5k+3}\equiv \left\{
           \begin{array}{ll}
           c_{10l+3}+d_{10l+3}, & \hbox{if $k=2l$;} \\
           c_{10l+8}+d_{10l+8}, & \hbox{if $k=2l+1$.}%
           \end{array}
           \right.$

           Since $10l+3,10l+8\leq 10k$, then, by the hypothesis, $c_{10l+3}\equiv
           c_{10l+8}\equiv 0,d_{10l+3}\equiv d_{10l+8}\equiv 1$. Hence, $%
           b_{10k+6}\equiv 1.$

           \item $b_{10k+8}\equiv c_{5k+4}+d_{5k+4}\equiv \left\{
           \begin{array}{ll}
           c_{10l+4}+d_{10l+4}, & \hbox{if $k=2l$;} \\
           c_{10l+9}+d_{10l+9}, & \hbox{if $k=2l+1$.}%
           \end{array}
           \right.$

           Since $10l+4,10l+9\leq 10k$, then, by the hypothesis, $c_{10l+4}\equiv
           d_{10l+9}\equiv 0,c_{10l+9}\equiv d_{10l+4}\equiv 1$. Hence, $%
            b_{10k+8}\equiv 1.$

           \item $b_{10k+10}\equiv c_{5k+5}+d_{5k+5}\equiv \left\{
           \begin{array}{ll}
           c_{10l+5}+d_{10l+5}, & \hbox{if $k=2l$;} \\
           c_{10l+10}+d_{10l+10}, & \hbox{if $k=2l+1$.}%
           \end{array}
           \right.$

           Since $10l+5,10l+10\leq 10k$, then, by the hypothesis, $c_{10l+5}\equiv
           d_{10l+5}\equiv 1,c_{10l+10}\equiv d_{10l+10}\equiv 0$. Hence, $%
           b_{10k+10}\equiv 0.$
           \end{itemize}

           Thus, the second assertion is true for $n\leq 10(k+1)$, which completes the proof  the second assertion.
  \item By the equalities (5) and (6) of Lemma \ref{lemma1}, we have following two subcases.
           \begin{itemize}
           \item $c_{10k+i}\equiv 0$ for $i=2,4,6,8,10$, since $c_{2n}\equiv0$ for $n\geq1$.

           \item Since $c_{2n+1}\equiv g_{n}+h_{n}\equiv a_{2n+1}$ for $n\geq1$. Hence,
           $c_{10k+i}\equiv a_{10k+i}\equiv 1$ for $i=1,5,9$ and $c_{10k+i}\equiv a_{10k+i}\equiv 0$ for $i=3,7$.
           \end{itemize}

           Thus, the third assertion is true for $n\leq 10(k+1)$, which completes the proof  the third assertion.
  \item By the equalities (7) and (8) of Lemma \ref{lemma1}, we have following six subcases.
           \begin{itemize}
           \item Since $d_{2n}\equiv c_{n}+d_{n}\equiv b_{2n}$ for $n\geq1$. Hence,
           $d_{10k+i}\equiv b_{10k+i}\equiv 1$ for $i=2,4,6,8$ and $d_{10k+10}\equiv b_{10k+10}\equiv 0$ .

           \item $d_{10k+1}\equiv x_{5k}+y_{5k}\equiv \left\{
           \begin{array}{ll}
           x_{10l}+y_{10l}, & \hbox{if $k=2l$;} \\
           x_{10l+5}+y_{10l+5}, & \hbox{if $k=2l+1$.}%
           \end{array}
           \right.$

           Since $10l,10l+5\leq 10k$, then, by the hypothesis, $x_{10l}\equiv
           y_{10l}\equiv 0,x_{10l+5}\equiv y_{10l+5}\equiv 1$. Hence, $d_{10k+1}\equiv0.$

           \item $d_{10k+3}\equiv x_{5k+1}+y_{5k+1}\equiv \left\{
           \begin{array}{ll}
           x_{10l+1}+y_{10l+1}, & \hbox{if $k=2l$;} \\
           x_{10l+6}+y_{10l+6}, & \hbox{if $k=2l+1$.}%
           \end{array}
           \right.$

           Since $10l+1,10l+6\leq 10k$, then, by the hypothesis, $x_{10l+1}\equiv
           y_{10l+6}\equiv 1,x_{10l+6}\equiv y_{10l+1}\equiv 0$. Hence, $%
           d_{10k+3}\equiv 1.$

           \item $d_{10k+5}\equiv x_{5k+2}+y_{5k+2}\equiv \left\{
            \begin{array}{ll}
           x_{10l+2}+y_{10l+2}, & \hbox{if $k=2l$;} \\
           x_{10l+7}+y_{10l+7}, & \hbox{if $k=2l+1$.}%
           \end{array}
           \right.$

           Since $10l+2,10l+7\leq 10k$, then, by the hypothesis, $x_{10l+2}\equiv
           x_{10l+7}\equiv 0,x_{10l+2}\equiv y_{10l+7}\equiv 1$. Hence, $%
           d_{10k+5}\equiv 1.$

           \item $d_{10k+7}\equiv x_{5k+3}+y_{5k+3}\equiv \left\{
            \begin{array}{ll}
           x_{10l+3}+y_{10l+3}, & \hbox{if $k=2l$;} \\
           x_{10l+8}+y_{10l+8}, & \hbox{if $k=2l+1$.}%
           \end{array}
           \right.$

           Since $10l+3,10l+8\leq 10k$, then, by the hypothesis, $x_{10l+3}\equiv
           y_{10l+8}\equiv 0,x_{10l+8}\equiv y_{10l+3}\equiv 1$. Hence, $%
           d_{10k+7}\equiv 1.$

           \item $d_{10k+9}\equiv x_{5k+4}+y_{5k+4}\equiv \left\{
            \begin{array}{ll}
           x_{10l+4}+y_{10l+4}, & \hbox{if $k=2l$;} \\
           x_{10l+9}+y_{10l+9}, & \hbox{if $k=2l+1$.}%
           \end{array}
           \right.$

           Since $10l+4,10l+9\leq 10k$, then, by the hypothesis, $x_{10l+4}\equiv
           y_{10l+4}\equiv 1,x_{10l+9}\equiv y_{10l+9}\equiv 0$. Hence, $%
           d_{10k+9}\equiv 0.$
           \end{itemize}

           Thus, the fourth assertion is true for $n\leq 10(k+1)$, which completes the proof  the fourth assertion.

\item By the equalities (9) and (10) of Lemma \ref{lemma1}, we have two subcases.
           \begin{itemize}
           \item Since $e_{2n}\equiv (a_{n}+b_{n})(c_{n}+d_{n})\equiv a_{2n}b_{2n}$, hence,
           $e_{10k+2}\equiv e_{10k+8}\equiv 1$ and $e_{10k+4}\equiv e_{10k+6}\equiv e_{10k+10}\equiv0$.

           \item Since $e_{2n+1}\equiv (g_{n}+h_{n})(x_{n}+y_{n})\equiv a_{2n+1}d_{2n+1}$, hence,
           $e_{10k+5}\equiv1$ and $e_{10k+1}\equiv e_{10k+3}\equiv e_{10k+7}\equiv e_{10k+9}\equiv0$.
           \end{itemize}

           Thus, the fifth assertion is true for $n\leq 10(k+1)$, which completes the proof  the fifth assertion.

\item By the equalities (11) and (12) of Lemma \ref{lemma1}, we have following ten subcases.
           \begin{itemize}
           \item $g_{10k+1}\equiv a_{5k+1}(g_{5k}+h_{5k})+x_{5k}(d_{5k+1}+e_{5k+1})+y_{5k}(c_{5k+1}+e_{5k+1})$. If $k=2l$, by the hypothesis, then $
           g_{10k+1} \equiv  a_{10l+1}(g_{10l}+h_{10l})+x_{10l}(d_{10l+1}+e_{10l+1})+y_{10l}(c_{10l+1}+e_{10l+1})
               \equiv  1\times(1+0)+0\times(0+0)+0\times(1+0) \equiv 1. $ If $k=2l+1$, by the hypothesis, then $
           g_{10k+1} \equiv  a_{10l+6}(g_{10l+5}+h_{10l+5})+x_{10l+5}(d_{10l+6}+e_{10l+6})+y_{10l+5}(c_{10l+6}+e_{10l+6})
               \equiv  0\times(0+1)+1\times(1+0)+1\times(0+0) \equiv 1. $

           \item $g_{10k+2}\equiv a_{5k+1}(g_{5k+1}+h_{5k+1})+x_{5k+1}(d_{5k+1}+e_{5k+1})+y_{5k+1}(c_{5k+1}+e_{5k+1})$. If $k=2l$, by the hypothesis, then $
           g_{10k+1} \equiv  a_{10l+1}(g_{10l+1}+h_{10l+1})+x_{10l+1}(d_{10l+1}+e_{10l+1})+y_{10l+1}(c_{10l+1}+e_{10l+1})
               \equiv  1\times(1+1)+1\times(0+0)+0\times(1+0) \equiv 0. $ If $k=2l+1$, by the hypothesis, then $
           g_{10k+1} \equiv  a_{10l+6}(g_{10l+6}+h_{10l+6})+x_{10l+6}(d_{10l+6}+e_{10l+6})+y_{10l+6}(c_{10l+6}+e_{10l+6})
               \equiv  0\times(0+0)+0\times(1+0)+1\times(0+0) \equiv 0. $

           \item $g_{10k+3}\equiv a_{5k+2}(g_{5k+1}+h_{5k+1})+x_{5k+1}(d_{5k+2}+e_{5k+2})+y_{5k+1}(c_{5k+2}+e_{5k+2})$. If $k=2l$, by the hypothesis, then $
           g_{10k+3} \equiv  a_{10l+2}(g_{10l+1}+h_{10l+1})+x_{10l+1}(d_{10l+2}+e_{10l+2})+y_{10l+1}(c_{10l+2}+e_{10l+2})
               \equiv  1\times(1+1)+1\times(1+1)+0\times(0+1) \equiv 0. $ If $k=2l+1$, by the hypothesis, then $
           g_{10k+3} \equiv  a_{10l+7}(g_{10l+6}+h_{10l+6})+x_{10l+6}(d_{10l+7}+e_{10l+7})+y_{10l+6}(c_{10l+7}+e_{10l+7})
               \equiv  0\times(0+0)+0\times(1+0)+1\times(0+0) \equiv 0. $

           \item $g_{10k+4}\equiv a_{5k+2}(g_{5k+2}+h_{5k+2})+x_{5k+2}(d_{5k+2}+e_{5k+2})+y_{5k+2}(c_{5k+2}+e_{5k+2})$. If $k=2l$, by the hypothesis, then $
           g_{10k+4} \equiv  a_{10l+2}(g_{10l+2}+h_{10l+2})+x_{10l+2}(d_{10l+2}+e_{10l+2})+y_{10l+2}(c_{10l+2}+e_{10l+2})
               \equiv  1\times(1+1)+1\times(0+0)+0\times(1+0) \equiv 0. $ If $k=2l+1$, by the hypothesis, then $
           g_{10k+4} \equiv  a_{10l+7}(g_{10l+7}+h_{10l+7})+x_{10l+7}(d_{10l+7}+e_{10l+7})+y_{10l+7}(c_{10l+7}+e_{10l+7})
               \equiv  0\times(0+0)+0\times(1+0)+1\times(0+0) \equiv 0. $

           \item $g_{10k+5}\equiv a_{5k+3}(g_{5k+2}+h_{5k+2})+x_{5k+2}(d_{5k+3}+e_{5k+3})+y_{5k+2}(c_{5k+3}+e_{5k+3})$. If $k=2l$, by the hypothesis, then $
           g_{10k+5} \equiv  a_{10l+3}(g_{10l+2}+h_{10l+2})+x_{10l+2}(d_{10l+3}+e_{10l+3})+y_{10l+2}(c_{10l+3}+e_{10l+3})
               \equiv  1\times(1+1)+1\times(1+1)+0\times(0+1) \equiv 0. $ If $k=2l+1$, by the hypothesis, then $
           g_{10k+5} \equiv  a_{10l+8}(g_{10l+7}+h_{10l+7})+x_{10l+7}(d_{10l+8}+e_{10l+8})+y_{10l+7}(c_{10l+8}+e_{10l+8})
               \equiv  0\times(0+0)+0\times(1+0)+1\times(0+0) \equiv 0. $

           \item $g_{10k+6}\equiv a_{5k+3}(g_{5k+3}+h_{5k+3})+x_{5k+3}(d_{5k+3}+e_{5k+3})+y_{5k+3}(c_{5k+3}+e_{5k+3})$. If $k=2l$, by the hypothesis, then $
           g_{10k+6} \equiv  a_{10l+3}(g_{10l+3}+h_{10l+3})+x_{10l+3}(d_{10l+3}+e_{10l+3})+y_{10l+3}(c_{10l+3}+e_{10l+3})
               \equiv  0\times(0+0)+0\times(1+0)+1\times(0+0) \equiv 0. $ If $k=2l+1$, by the hypothesis, then $
           g_{10k+6} \equiv  a_{10l+8}(g_{10l+8}+h_{10l+8})+x_{10l+8}(d_{10l+8}+e_{10l+8})+y_{10l+8}(c_{10l+8}+e_{10l+8})
               \equiv  1\times(1+1)+1\times(1+1)+0\times(0+1) \equiv 0. $

           \item $g_{10k+7}\equiv a_{5k+4}(g_{5k+3}+h_{5k+3})+x_{5k+3}(d_{5k+4}+e_{5k+4})+y_{5k+3}(c_{5k+4}+e_{5k+4})$. If $k=2l$, by the hypothesis, then $
           g_{10k+7} \equiv  a_{10l+4}(g_{10l+3}+h_{10l+3})+x_{10l+3}(d_{10l+4}+e_{10l+4})+y_{10l+3}(c_{10l+4}+e_{10l+4})
               \equiv  0\times(0+0)+0\times(1+0)+1\times(0+0) \equiv 0. $ If $k=2l+1$, by the hypothesis, then $
           g_{10k+7} \equiv  a_{10l+9}(g_{10l+8}+h_{10l+8})+x_{10l+8}(d_{10l+9}+e_{10l+9})+y_{10l+8}(c_{10l+9}+e_{10l+9})
               \equiv  1\times(1+1)+1\times(0+0)+0\times(1+0) \equiv 0. $

           \item $g_{10k+8}\equiv a_{5k+4}(g_{5k+4}+h_{5k+4})+x_{5k+4}(d_{5k+4}+e_{5k+4})+y_{5k+4}(c_{5k+4}+e_{5k+4})$. If $k=2l$, by the hypothesis, then $
           g_{10k+8} \equiv  a_{10l+4}(g_{10l+4}+h_{10l+4})+x_{10l+4}(d_{10l+4}+e_{10l+4})+y_{10l+4}(c_{10l+4}+e_{10l+4})
               \equiv  0\times(0+1)+1\times(1+0)+1\times(0+0) \equiv 1. $ If $k=2l+1$, by the hypothesis, then $
           g_{10k+8} \equiv  a_{10l+9}(g_{10l+9}+h_{10l+9})+x_{10l+9}(d_{10l+9}+e_{10l+9})+y_{10l+9}(c_{10l+9}+e_{10l+9})
               \equiv  1\times(1+0)+0\times(0+0)+0\times(1+0) \equiv 1. $

           \item $g_{10k+9}\equiv a_{5k+5}(g_{5k+4}+h_{5k+4})+x_{5k+4}(d_{5k+5}+e_{5k+5})+y_{5k+4}(c_{5k+5}+e_{5k+5})$. If $k=2l$, by the hypothesis, then $
           g_{10k+9} \equiv  a_{10l+5}(g_{10l+4}+h_{10l+4})+x_{10l+4}(d_{10l+5}+e_{10l+5})+y_{10l+4}(c_{10l+5}+e_{10l+5})
               \equiv  1\times(0+1)+1\times(1+1)+1\times(1+1) \equiv 1. $ If $k=2l+1$, by the hypothesis, then $
           g_{10k+9} \equiv  a_{10l+10}(g_{10l+9}+h_{10l+9})+x_{10l+9}(d_{10l+10}+e_{10l+10})+y_{10l+9}(c_{10l+10}+e_{10l+10})
               \equiv  1\times(1+0)+0\times(0+0)+0\times(0+0) \equiv 1. $

           \item $g_{10k+10}\equiv a_{5k+5}(g_{5k+5}+h_{5k+5})+x_{5k+5}(d_{5k+5}+e_{5k+5})+y_{5k+5}(c_{5k+5}+e_{5k+5})$. If $k=2l$, by the hypothesis, then $
           g_{10k+10} \equiv  a_{10l+5}(g_{10l+5}+h_{10l+5})+x_{10l+5}(d_{10l+5}+e_{10l+5})+y_{10l+5}(c_{10l+5}+e_{10l+5})
               \equiv  1\times(0+1)+1\times(1+1)+1\times(1+1) \equiv 1. $ If $k=2l+1$, by the hypothesis, then $
           g_{10k+10} \equiv  a_{10l+10}(g_{10l+10}+h_{10l+10})+x_{10l+10}(d_{10l+10}+e_{10l+10})+y_{10l+10}(c_{10l+10}+e_{10l+10})
               \equiv  1\times(1+0)+0\times(0+0)+0\times(0+0) \equiv 1. $
           \end{itemize}

          Thus, the sixth assertion is true for $n\leq 10(k+1)$, which completes the proof  the sixth assertion.

\item By the equalities (13) and (14) of Lemma \ref{lemma1}, we have following ten subcases.
           \begin{itemize}
           \item $h_{10k+1}\equiv g_{5k}(c_{5k+1}+e_{5k+1})+h_{5k}(d_{5k+1}+e_{5k+1})+b_{5k+1}(x_{5k}+y_{5k})$. If $k=2l$, by the hypothesis, then $
           h_{10k+1} \equiv   g_{10l}(c_{10l+1}+e_{10l+1})+h_{10l}(d_{10l+1}+e_{10l+1})+b_{10l+1}(x_{10l}+y_{10l})
               \equiv  1\times(1+0)+0\times(0+0)+0\times(0+0) \equiv 1. $ If $k=2l+1$, by the hypothesis, then $
           h_{10k+1} \equiv  g_{10l+5}(c_{10l+6}+e_{10l+6})+h_{10l+5}(d_{10l+6}+e_{10l+6})+b_{10l+6}(x_{10l+5}+y_{10l+5})
               \equiv  0\times(0+0)+1\times(1+0)+1\times(1+1) \equiv 1. $

           \item $h_{10k+2}\equiv g_{5k+1}(c_{5k+1}+e_{5k+1})+h_{5k+1}(d_{5k+1}+e_{5k+1})+b_{5k+1}(x_{5k+1}+y_{5k+1})$. If $k=2l$, by the hypothesis, then $
           h_{10k+2} \equiv  g_{10l+1}(c_{10l+1}+e_{10l+1})+h_{10l+1}(d_{10l+1}+e_{10l+1})+b_{10l+1}(x_{10l+1}+y_{10l+1})
               \equiv  1\times(1+0)+1\times(0+0)+0\times(1+0) \equiv 1. $ If $k=2l+1$, by the hypothesis, then $
           h_{10k+2} \equiv  g_{10l+6}(c_{10l+6}+e_{10l+6})+h_{10l+6}(d_{10l+6}+e_{10l+6})+b_{10l+6}(x_{10l+6}+y_{10l+6})
               \equiv  0\times(0+0)+0\times(1+0)+1\times(0+1) \equiv 1. $

           \item $h_{10k+3}\equiv g_{5k+1}(c_{5k+2}+e_{5k+2})+h_{5k+1}(d_{5k+2}+e_{5k+2})+b_{5k+2}(x_{5k+1}+y_{5k+1})$. If $k=2l$, by the hypothesis, then $
           h_{10k+3} \equiv  g_{10l+1}(c_{10l+2}+e_{10l+2})+h_{10l+1}(d_{10l+2}+e_{10l+2})+b_{10l+2}(x_{10l+1}+y_{10l+1})
               \equiv  1\times(0+1)+1\times(1+1)+1\times(1+0) \equiv 0. $ If $k=2l+1$, by the hypothesis, then $
           h_{10k+3} \equiv  g_{10l+6}(c_{10l+7}+e_{10l+7})+h_{10l+6}(d_{10l+7}+e_{10l+7})+b_{10l+7}(x_{10l+6}+y_{10l+6})
               \equiv  0\times(0+0)+0\times(1+0)+0\times(0+1) \equiv 0. $

           \item $h_{10k+4}\equiv g_{5k+2}(c_{5k+2}+e_{5k+2})+h_{5k+2}(d_{5k+2}+e_{5k+2})+b_{5k+2}(x_{5k+2}+y_{5k+2})$. If $k=2l$, by the hypothesis, then $
           h_{10k+4} \equiv  g_{10l+2}(c_{10l+2}+e_{10l+2})+h_{10l+2}(d_{10l+2}+e_{10l+2})+b_{10l+2}(x_{10l+2}+y_{10l+2})
               \equiv  0\times(0+1)+1\times(1+1)+1\times(0+1) \equiv 1. $ If $k=2l+1$, by the hypothesis, then $
           h_{10k+4} \equiv  g_{10l+7}(c_{10l+7}+e_{10l+7})+h_{10l+7}(d_{10l+7}+e_{10l+7})+b_{10l+7}(x_{10l+7}+y_{10l+7})
               \equiv  0\times(0+0)+0\times(1+0)+1\times(0+1) \equiv 1. $

           \item $h_{10k+5}\equiv g_{5k+2}(c_{5k+3}+e_{5k+3})+h_{5k+2}(d_{5k+3}+e_{5k+3})+b_{5k+3}(x_{5k+2}+y_{5k+2})$. If $k=2l$, by the hypothesis, then $
           h_{10k+5} \equiv  g_{10l+2}(c_{10l+3}+e_{10l+3})+h_{10l+2}(d_{10l+3}+e_{10l+3})+b_{10l+3}(x_{10l+2}+y_{10l+2})
               \equiv  0\times(0+0)+1\times(1+0)+0\times(0+1) \equiv 1. $ If $k=2l+1$, by the hypothesis, then $
           h_{10k+5} \equiv  g_{10l+7}(c_{10l+8}+e_{10l+8})+h_{10l+7}(d_{10l+8}+e_{10l+8})+b_{10l+8}(x_{10l+7}+y_{10l+7})
               \equiv  0\times(0+1)+1\times(1+1)+1\times(0+1) \equiv 1. $

           \item $h_{10k+6}\equiv g_{5k+3}(c_{5k+3}+e_{5k+3})+h_{5k+3}(d_{5k+3}+e_{5k+3})+b_{5k+3}(x_{5k+3}+y_{5k+3})$. If $k=2l$, by the hypothesis, then $
           h_{10k+6} \equiv  g_{10l+3}(c_{10l+3}+e_{10l+3})+h_{10l+3}(d_{10l+3}+e_{10l+3})+b_{10l+3}(x_{10l+3}+y_{10l+3})
               \equiv  0\times(0+0)+0\times(1+0)+0\times(0+1) \equiv 0. $ If $k=2l+1$, by the hypothesis, then $
           h_{10k+6} \equiv  g_{10l+8}(c_{10l+8}+e_{10l+8})+h_{10l+8}(d_{10l+8}+e_{10l+8})+b_{10l+8}(x_{10l+8}+y_{10l+8})
               \equiv  1\times(0+1)+1\times(1+1)+1\times(1+0) \equiv 0. $

           \item $h_{10k+7}\equiv g_{5k+3}(c_{5k+4}+e_{5k+4})+h_{5k+3}(d_{5k+4}+e_{5k+4})+b_{5k+4}(x_{5k+3}+y_{5k+3})$. If $k=2l$, by the hypothesis, then $
           h_{10k+7} \equiv   g_{10l+3}(c_{10l+4}+e_{10l+4})+h_{10l+3}(d_{10l+4}+e_{10l+4})+b_{10l+4}(x_{10l+3}+y_{10l+3})
               \equiv  0\times(0+0)+0\times(1+0)+1\times(0+1) \equiv 1. $ If $k=2l+1$, by the hypothesis, then $
           h_{10k+7} \equiv  g_{10l+8}(c_{10l+9}+e_{10l+9})+h_{10l+8}(d_{10l+9}+e_{10l+9})+b_{10l+9}(x_{10l+8}+y_{10l+8})
               \equiv  1\times(1+0)+1\times(0+0)+0\times(1+0) \equiv 1. $

           \item $h_{10k+8}\equiv g_{5k+4}(c_{5k+4}+e_{5k+4})+h_{5k+4}(d_{5k+4}+e_{5k+4})+b_{5k+4}(x_{5k+4}+y_{5k+4})$. If $k=2l$, by the hypothesis, then $
           h_{10k+8} \equiv   g_{10l+4}(c_{10l+4}+e_{10l+4})+h_{10l+4}(d_{10l+4}+e_{10l+4})+b_{10l+4}(x_{10l+4}+y_{10l+4})
               \equiv  0\times(0+0)+1\times(1+0)+1\times(1+1) \equiv 1. $ If $k=2l+1$, by the hypothesis, then $
           h_{10k+8} \equiv  g_{10l+9}(c_{10l+9}+e_{10l+9})+h_{10l+9}(d_{10l+9}+e_{10l+9})+b_{10l+9}(x_{10l+9}+y_{10l+9})
               \equiv  1\times(1+0)+0\times(0+0)+0\times(0+0) \equiv 1. $

           \item $h_{10k+9}\equiv g_{5k+4}(c_{5k+5}+e_{5k+5})+h_{5k+4}(d_{5k+5}+e_{5k+5})+b_{5k+5}(x_{5k+4}+y_{5k+4})$. If $k=2l$, by the hypothesis, then $
           h_{10k+9} \equiv  g_{10l+4}(c_{10l+5}+e_{10l+5})+h_{10l+4}(d_{10l+5}+e_{10l+5})+b_{10l+5}(x_{10l+4}+y_{10l+4})
               \equiv  0\times(1+1)+1\times(1+1)+0\times(1+1) \equiv 0. $ If $k=2l+1$, by the hypothesis, then $
           h_{10k+9} \equiv  g_{10l+9}(c_{10l+10}+e_{10l+10})+h_{10l+9}(d_{10l+10}+e_{10l+10})+b_{10l+10}(x_{10l+9}+y_{10l+9})
               \equiv  1\times(0+0)+0\times(0+0)+0\times(0+0) \equiv 0. $

           \item $h_{10k+10}\equiv g_{5k+5}(c_{5k+5}+e_{5k+5})+h_{5k+5}(d_{5k+5}+e_{5k+5})+b_{5k+5}(x_{5k+5}+y_{5k+5})$. If $k=2l$, by the hypothesis, then $
           h_{10k+10} \equiv  g_{10l+5}(c_{10l+5}+e_{10l+5})+h_{10l+5}(d_{10l+5}+e_{10l+5})+b_{10l+5}(x_{10l+5}+y_{10l+5})
               \equiv  0\times(1+1)+1\times(1+1)+0\times(1+1) \equiv 0. $ If $k=2l+1$, by the hypothesis, then $
           h_{10k+10} \equiv  g_{10l+10}(c_{10l+10}+e_{10l+10})+h_{10l+10}(d_{10l+10}+e_{10l+10})+b_{10l+10}(x_{10l+10}+y_{10l+10})
               \equiv  1\times(0+0)+0\times(0+0)+0\times(0+0) \equiv 0. $
           \end{itemize}

          Thus, the seventh assertion is true for $n\leq 10(k+1)$, which completes the proof  the seventh assertion.

\item By the equalities (15) and (16) of Lemma \ref{lemma1}, we have two subcases.
           \begin{itemize}
           \item Since $x_{2n}\equiv (g_{n}+h_{n})(c_{n}+d_{n})\equiv a_{2n+1}b_{2n}$, hence,
           $x_{10k+4}\equiv x_{10k+8}\equiv 1$ and $x_{10k+2}\equiv x_{10k+6}\equiv x_{10k+10}\equiv0$.

           \item Since $x_{2n+1}\equiv (g_{n}+h_{n})(c_{n+1}+d_{n+1})\equiv a_{2n+1}b_{2n+2}$, hence,
           $x_{10k+1}\equiv x_{10k+5}\equiv1$ and $x_{10k+3}\equiv x_{10k+7}\equiv x_{10k+9}\equiv0$.
           \end{itemize}

           Thus, the eighth assertion is true for $n\leq 10(k+1)$, which completes the proof  the eighth assertion.

\item By the equalities (17) and (18) of Lemma \ref{lemma1}, we have two subcases.
           \begin{itemize}
           \item Since $y_{2n}\equiv (x_{n}+y_{n})(c_{n}+d_{n})\equiv d_{2n+1}b_{2n}$, hence,
           $y_{10k+2}\equiv y_{10k+4}\equiv y_{10k+6}\equiv 1$ and $y_{10k+8}\equiv y_{10k+10}\equiv 0$.

           \item Since $y_{2n+1}\equiv (x_{n}+y_{n})(c_{n+1}+d_{n+1})\equiv d_{2n+1}b_{2n+2}$, hence,
           $y_{10k+3}\equiv y_{10k+5}\equiv y_{10k+7}\equiv1$ and $y_{10k+1}\equiv y_{10k+9}\equiv0$.
           \end{itemize}

           Thus, the ninth assertion is true for $n\leq 10(k+1)$, which completes the proof  the ninth assertion.

\end{enumerate}
\end{proof}

\section{Rational approximation}

In this section, we will use Pad\'{e} approximant to construct an infinite
sequence of `good' rational approximations. The method we use here, was
introduced in \cite{B80}. And, we prove Theorem \ref{proposition1} at
the end of this section.

Given a generating function, say $f(z)=\sum\limits_{i=0}^{\infty}u_{i}z^{i}$%
, its \emph{Pad\'{e} approximant }$[p/q]_{f}$ is a rational function $\frac{%
P(z)}{Q(z)}$ with $\deg(P(z))\leq p,\deg(Q(z))\leq q$ satisfying
\begin{eqnarray*}
f(z)=\frac{P(z)}{Q(z)}+\mathcal{O}(z^{p+q+1}), ~z\rightarrow0.
\end{eqnarray*}
where $p,q \in\mathbb{N}.$

By a classical result, if $H_{k}(f)$ is nonzero, then the pad\'{e}
approximant $[k-1/k]_{f}(z)$ exists (see \cite{B80}, pp. 34-36) and
\begin{eqnarray}  \label{formula7}
f(z)-[k-1/k]_{f}(z)=\frac{H_{k+1}(f)}{H_{k}(f)}z^{2k}+\mathcal{O}(z^{2k+1}).
\end{eqnarray}

By Lemma \ref{lemma2}, we have following corollary.

\begin{corollary}
\label{cor1} For any $k\equiv 1 (\bmod 10)$, the pad\'{e} approximant $%
[k-1/k]_{F}(z)$ exists. Moreover, there exist a nonzero real number $h_{k}$
such that
\begin{equation*}
F(z)-[k-1/k]_{F}(z)=h_{k}z^{2k}+\mathcal{O}(z^{2k+1}).
\end{equation*}
\end{corollary}

\begin{proof}
By Lemma \ref{lemma2}, this is an immediate conclusion of the fact that $H_{k}(F)\equiv H_{k+1}(F)\equiv 1$ if $k\equiv 1 (%
\bmod 10)$, then, taking $h_{k}=\frac{H_{k+1}(F)}{H_{k}(F)}$.
\end{proof}

To construct an infinite sequence of `good' rational approximations, we need
following two lemmas.
%One is a result of Adamczewski and Rivoal (\cite{AR09}%
%, Lemma 4.1), and the other is a slight modification of Lemma 2 of \cite%
%${Bugeaud11}.

\begin{lemma}[Adamaczewski and Rivoal \protect\cite{AR09}]
\label{lemma3} Let $\xi,\delta,\rho,\theta\in\mathbb{R}$ and $\delta\leq\rho$%
, $\theta>1$. Assume that there exist positive real numbers $c_{0},c_{1}\leq
c_{2}$ and a sequence $(\frac{p_{n}}{q_{n}})_{n\geq1}$ of rational numbers
such that
\begin{equation*}
q_{n}<q_{n+1}\leq c_{0}q_{n}^{\theta},(n\geq1),
\end{equation*}
and
\begin{equation*}
\frac{c_{1}}{q_{n}^{1+\rho}}\leq|\xi-\frac{p_{n}}{q_{n}}|\leq\frac{c_{2}}{%
q_{n}^{1+\delta}},(n\geq1).
\end{equation*}
Then we have $\mu(\xi)\leq(1+\rho)\frac{\theta}{\delta}$.
\end{lemma}

\begin{lemma}[Modified Bugeaud \protect\cite{Bugeaud11}]
\label{lemma4} Let $L,k,m_{0}\geq2$ be positive integers, $R:=R(L)$ be a
real number. Let $\mathcal{A}$ be any subset of integers of $%
[k^{L-1},k^{L}-1]$ satisfying $[k^{L-1},k^{L}-1]\subset\bigcup_{x\in\mathcal{%
A}}[x-R,x+R]$. Assume $\{a_{j}\}_{j\geq 0}$ is an increasing sequence
composed of all the numbers of the form $lk^{m}$, where $m\geq m_{0}$ and $l$
ranges over $\mathcal{A}$. Then, there exists a positive integer $j_{0}$,
such that for all $j\geq j_{0}$,
\begin{equation*}
\frac{a_{j+1}}{a_{j}}<1+\frac{(k+1)R}{ k^{L-1}}.
\end{equation*}
\end{lemma}

\begin{proof}
Suppose $\mathcal{A}=\{n_{0},n_{1},\cdots,n_{t}\}$, then $0\leq
n_{i+1}-n_{i}\leq 2R$ for $0\leq i\leq t-1$, and $kn_{0}-n_{t}\leq(k+1)R$.
Taking $j$ large enough, then we have
\begin{eqnarray*}
\frac{a_{j+1}}{a_{j}}&\leq& \max\left\{\max_{0\leq i\leq t}\{\frac{n_{i+1}}{%
n_{i}}\},\frac{k n_{0}}{n_{t}}\right\} \\
&\leq&\max_{k^{L-1}\leq n_{0}\leq k^{L}-1}\left\{1+\frac{(k+1)R}{n_{0}}\right\}\\
&=&1+\frac{(k+1)R}{ k^{L-1}},
\end{eqnarray*}
which ends the proof.
\end{proof}

\iffalse
Let $\{a_{i}\}_{i\geq 0}$ be a sequence over $\mathbb{Z}$, whose generating
function is $f(x)=\sum_{i\geq 0}a_{i}x^{i}$. Define the Hankel determinant
of order $n$ associated to $f(x)$ by
\begin{equation*}
H_{n}(f):=\left\vert
\begin{array}{cccc}
a_{0} & a_{1} & \cdots & a_{n-1} \\
a_{1} & a_{2} & \cdots & a_{n} \\
\vdots & \vdots & \ddots & \vdots \\
a_{n-1} & a_{n} & \cdots & a_{2n-2}%
\end{array}%
\right\vert
\end{equation*}
\fi
Now, we are going to prove Theorem \ref{proposition1}.

\begin{proof}[Proof of Theorem \protect\ref{proposition1}.]
Since $f(x)=\frac{A(x)}{B(x)}+C(x)f(x^{k})$, we have for any $m\geq 2$%
\begin{equation}\label{equation20}
f(x)=\frac{A_{m}(x)}{B_{m}(x)}+C_{m}(x)f(x^{k^{m}}).
\end{equation}%
where
\begin{eqnarray*}
&&A_{m}(x):= A(x)\frac{B_{m}(x)}{B(x)}+\sum\limits_{j=1}^{m-1}\prod%
\limits_{i=0}^{j-1}C(x^{k^{i}})A(x^{k^{j}})\frac{B_{m}(x)}{B(x^{k^{j}})}, \\
&&B_{m}(x) :=\prod\limits_{j=0}^{m-1}B(x^{k^{j}}),
\quad C_{m}(x):=\prod\limits_{j=0}^{m-1}C(x^{k^{j}}).
\end{eqnarray*}

Set $\deg (A(x))=\alpha $, $\deg (B(x))=\beta $, $\deg (C(x))=\gamma $.
Then, we have
\begin{eqnarray*}
\deg (A_{m}(x)) &=&\max \left\{ \alpha +\beta \frac{k^{m}-1}{k-1}-\beta
,\right. \\
&&\left. \max_{1\leq j\leq m-1}\{\gamma \frac{k^{j}-1}{k-1}+\alpha
k^{j}+\beta \frac{k^{m}-1}{k-1}-\beta k^{j}\}\right\} \\
&=&\max \left\{ \alpha +\beta \frac{k^{m}-1}{k-1}-\beta ,\gamma +\alpha
k+\beta \frac{k^{m}-1}{k-1}-\beta k,\right. \\
&&\left. \gamma \frac{k^{m-1}-1}{k-1}+\alpha k^{m-1}+\beta \frac{k^{m-1}-1}{%
k-1}\right\} \\
&\leq &(\alpha +\beta +\gamma )k^{m}, \\
\deg (B_{m}(x)) &=&\deg \left( \prod\limits_{j=0}^{m-1}B(x^{k^{j}})\right)
=\beta \sum_{0\leq i\leq m-1}k^{i}\leq \beta k^{m}, \\
\deg (C_{m}(x)) &=&\deg \left( \prod\limits_{j=0}^{m-1}C(x^{k^{j}})\right)
=\gamma \sum_{0\leq i\leq m-1}k^{i}\leq \gamma k^{m}. \\
&&
\end{eqnarray*}

Assume $H_{l}(f)H_{l+1}(f)\neq 0$, by the formula (\ref{formula7}), then
\begin{equation*}
f(x)-\frac{P_{l}(x)}{Q_{l}(x)}=h_{l}x^{2l}+\mathcal{O}(x^{2l+1}),
\end{equation*}
where $h_{l}\neq 0$ and $P_{l}(x),Q_{l}(x)\in \mathbb{Z}[x]$ with $\deg
(P_{l}(x))\leq l-1,\deg (Q_{l}(x))\leq l$. Then for any $m\geq 2$,
\begin{equation*}
f(x^{k^{m}})-\frac{P_{l}(x^{k^{m}})}{Q_{l}(x^{k^{m}})}=h_{l}x^{2l\cdot
k^{m}}+\mathcal{O}(x^{(2l+1)k^{m}}).
\end{equation*}

Then, there exist a positive constant $c(l)$ such that
\begin{equation*}
    \left|f(x^{k^{m}})-\frac{P_{l}(x^{k^{m}})}{Q_{l}(x^{k^{m}})}-h_{l}x^{2l\cdot
k^{m}}\right|\leq c(l)x^{(2l+1)k^{m}}, ~\text{for}~0<x\leq \frac{1}{2}.
\end{equation*}

Hence, for any $0<x\leq \frac{1}{2}$, by the formula (\ref{equation20}),
\begin{equation}\label{equation21}
    \left|f(x)-\frac{A_{m}(x)}{B_{m}(x)}-C_{m}(x)\cdot \frac{P_{l}(x^{k^{m}})}{%
Q_{l}(x^{k^{m}})}-C_{m}(x)h_{l}x^{2l\cdot
k^{m}}\right|\leq c(l)x^{(2l+1)k^{m}}|C_{m}(x)|.
\end{equation}

For simplicity, set
\begin{eqnarray*}
P_{l,m}(x)&:= &A_{m}(x) Q_{l}(x^{k^{m}})+B_{m}(x) C_{m}(x)
P_{l}(x^{k^{m}}), \\
Q_{l,m}(x)&:= &B_{m}(x)Q_{l}(x^{k^{m}}).
\end{eqnarray*}

We have
\begin{equation}\label{equation22}
    \left|f(x)-\frac{P_{l,m}(x)}{Q_{l,m}(x)}-C_{m}(x)h_{l}x^{2l\cdot
k^{m}}\right|\leq c(l)x^{(2l+1)k^{m}}|C_{m}(x)|.
\end{equation}
and
\begin{equation*}
       \begin{array}{rcl}
         \deg (P_{l,m}(x))  & = & \max \{\deg (A_{m}(x))+\deg (Q_{l}(x^{k^{m}})) \\
          &  & \qquad\deg (B_{m}(x))+\deg (C_{m}(x))+\deg (P_{l}(x^{k^{m}}))\} \\
          & \leq & \max \{(\alpha +\beta +\gamma )k^{m}+lk^{m}, \beta k^{m}+\gamma k^{m}+(l-1)k^{m}\} \\
          & \leq & (\alpha +\beta +\gamma +l)k^{m},\\
        \deg (Q_{l,m}(x)) &=& \deg (B_{m}(x))+\deg (Q_{l}(x^{k^{m}}))\leq \beta k^{m}+lk^{m}.
       \end{array}
\end{equation*}

Assume $C(x)=(a_{0}+a_{1}x+\cdots +a_{\gamma }x^{r})x^{s}$, where $a_{0}\neq0,r,s\geq0,r+s=\gamma$, and $\eta :=\max
\{a_{i}:0\leq i\leq \gamma \}$, then, for any $0<x\leq\frac{1}{2}$,
\begin{equation}\label{equation23}
    |C_{m}(x)|\leq\prod\limits_{j=0}^{m-1}x^{sk^{j}}\cdot\eta^{m}(1+x+x^{2}+\cdots)\leq 2\eta^{m}x^{\frac{s(k^{m}-1)}{k-1}}.
\end{equation}

Note that $C(\frac{1}{b})\neq0$ which implies that $C_{m}(\frac{1}{b})\neq0$. Hence, for any $x=\frac{1}{b}(b\geq2)$, there exists a positive constant $\zeta$ such that $|C(x)|\geq \zeta x^{s}$,
\begin{equation}\label{equation24}
    |C_{m}(x)|\geq  \zeta^{m}x^{\frac{s(k^{m}-1)}{k-1}}.
\end{equation}

Let $m\geq m_{0}(l)$ be large enough such that
\begin{equation}\label{equation25}
 c(l)2^{-k^{m}}\leq \frac{1}{2}|h_{l}|.
\end{equation}

By the formulae (\ref{equation22}) and (\ref{equation23})(\ref{equation24})(\ref{equation25}), for any $x=\frac{1}{b}(b\geq2)$,
\begin{equation}\label{equation26}
    \frac{1}{2}|h_{l}|\cdot \zeta^{m}x^{2lk^{m}+\frac{s(k^{m}-1)}{k-1}}\leq\left|f(x)-\frac{P_{l,m}(x)}{Q_{l,m}(x)}\right|\leq 3|h_{l}|\cdot \eta^{m}x^{2lk^{m}+\frac{s(k^{m}-1)}{k-1}}.
\end{equation}

\iffalse

Assume $C(x)=a_{0}+a_{1}x+\cdots +a_{\gamma }x^{\gamma }$ and $\eta :=\min
\{i:a_{i}\neq 0,0\leq i\leq \gamma \}$, we have
\begin{equation*}
C_{m}(x)\cdot f(x^{k^{m}})-C_{m}(x)\cdot \frac{P_{l}(x^{k^{m}})}{%
Q_{l}(x^{k^{m}})}=h_{l}x^{2l\cdot k^{m}+\eta \cdot \frac{k^{m}-1}{k-1}}+%
\mathcal{O}(x^{2l\cdot k^{m}+\eta \cdot \frac{k^{m}-1}{k-1}+1}).
\end{equation*}

Thus,
\begin{equation*}
f(x)-\frac{A_{m}(x)}{B_{m}(x)}-C_{m}(x)\cdot \frac{P_{l}(x^{k^{m}})}{%
Q_{l}(x^{k^{m}})}=h_{l}x^{2l\cdot k^{m}+\eta \cdot \frac{k^{m}-1}{k-1}}+%
\mathcal{O}(x^{2l\cdot k^{m}+\eta \cdot \frac{k^{m}-1}{k-1}+1}).
\end{equation*}

For simplicity, set
\begin{eqnarray*}
P_{l,m}(x)&:= &A_{m}(x) Q_{l}(x^{k^{m}})+B_{m}(x) C_{m}(x)
P_{l}(x^{k^{m}}), \\
Q_{l,m}(x)&:= &B_{m}(x)Q_{l}(x^{k^{m}}).
\end{eqnarray*}

Hence,
\begin{equation}  \label{formula8}
f(x)-\frac{P_{l,m}(x)}{Q_{l,m}(x)}=h_{l}x^{2l\cdot k^{m}+\eta \cdot \frac{%
k^{m}-1}{k-1}}+\mathcal{O}(x^{2l\cdot k^{m}+\eta \cdot \frac{k^{m}-1}{k-1}%
+1}).
\end{equation}%
and
\begin{equation*}
       \begin{array}{rcl}
         \deg (P_{l,m}(x))  & = & \max \{\deg (A_{m}(x))+\deg (Q_{l}(x^{k^{m}})) \\
          &  & \qquad\deg (B_{m}(x))+\deg (C_{m}(x))+\deg (P_{l}(x^{k^{m}}))\} \\
          & \leq & \max \{(\alpha +\beta +\gamma )k^{m}+lk^{m}, \beta k^{m}+\gamma k^{m}+(l-1)k^{m}\} \\
          & \leq & (\alpha +\beta +\gamma +l)k^{m},\\
        \deg (Q_{l,m}(x)) &=& \deg (B_{m}(x))+\deg (Q_{l}(x^{k^{m}}))\leq \beta k^{m}+lk^{m}.
       \end{array}
\end{equation*}
\fi

Suppose that $Y_{l}:=\alpha +\beta +\gamma +l$. Let $b\geq 2$ be an integer,
taking $x=1/b$, then both $p_{l,m}:=b^{Y_{l}k^{m}}P_{l,m}(1/b)$ and $%
q_{l,m}:=b^{Y_{l}k^{m}}Q_{l,m}(1/b)$ are integers. Moreover, since $%
Q_{l}(0)\neq 0,$ for sufficiently large $m$, there exist two positive real
numbers $c_{0},c_{1}$ depending only on $l,$ such that
\begin{equation}  \label{formula9}
c_{0}b^{Y_{l}k^{m}}\leq q_{l,m}\leq c_{1}b^{Y_{l}k^{m}}.
\end{equation}

which implies that
\begin{equation*}
\frac{c_{0}}{c_{1}}q_{l,m}<q_{l,m+1}\leq \frac{c_{1}}{c_{0}}q_{l,m}^{k}.
\end{equation*}

On the other hand, let $m$ is larger enough, then by (\ref{equation26}),
\begin{equation}  \label{formula10}
\frac{c_{2}}{\left( b^{Y_{l}k^{m}}\right) ^{\rho _{l}}}=\frac{c_{2}}{%
b^{(2l+\gamma )k^{m}}}\leq \left\vert f(\frac{1}{b})-\frac{p_{l,m}}{q_{l,m}}%
\right\vert \leq \frac{c_{3}}{b^{2lk^{m}}}=\frac{c_{3}}{\left(
b^{Y_{l}k^{m}}\right) ^{\delta _{l}}},
\end{equation}%
where the constants $c_{2},c_{3}$ depend on $l,$ and $\rho _{l}=\frac{%
2l+\gamma }{Y_{l}},\delta _{l}=\frac{2l}{Y_{l}}.$ Note that $\lim_{l\rightarrow\infty}\rho _{l}=\lim_{l\rightarrow\infty}\delta _{l}=2$.
By Lemma \ref{lemma3}, taking $l=n_{i}$, we have%
\begin{equation*}
\mu \left( f(\frac{1}{b})\right) \leq k\cdot \frac{\rho _{l}}{\delta _{l}-1}=k\cdot \frac{\rho _{n_{i}}}{\delta _{n_{i}}-1}\longrightarrow 2k,\quad
i\rightarrow\infty.
\end{equation*}

Now, assume that $\liminf\limits_{i\rightarrow \infty }\frac{n_{i+1}}{n_{i}}%
=1.$ Thus, for any $\varepsilon >0,$ we can find a subsequece $\left\{
n_{i_{j}}\right\} _{j\geq 0}$ of $\left\{ n_{i}\right\} _{i\geq 0}$ such
that for all $j\geq J(\varepsilon ),$%
\begin{equation*}
\frac{n_{i_{j+1}}}{n_{i_{j}}}<1+\varepsilon .
\end{equation*}
\iffalse
Give a sufficient large integer $L$ such that $k^{L-1}>n_{i_{J}}$. In this
case, $L\rightarrow \infty ,$ when $\varepsilon \rightarrow 0.$ Let $%
\{V_{L,m^{^{\prime }}}\}_{m^{^{\prime }}\geq 1}$ be the increasing sequence
composed of all the numbers $q_{l,m}$, where $m\geq m_{0}(l)$ and $l$ ranges
over the integers in $\{n_{i_{j}}:n_{i_{j}}\in \lbrack k^{L-1},k^{L}-1]\}$.
\fi

Give a sufficient large integer $L$ such that $k^{L-1}>n_{i_{J}}$. In this
case, $L\rightarrow \infty ,$ when $\varepsilon \rightarrow 0.$ Let $%
\{V_{L,m^{^{\prime }}}\}_{m^{^{\prime }}\geq 1}$ be the increasing sequence
composed of all the numbers $q_{l,m}$, where $m\geq m_{0}(l)$ and $l$ ranges
over the integers in $\mathcal{A}=\{n: \exists~j~\text{such that}~n=n_{i_{j}},n\in
\lbrack k^{L-1},k^{L}-1]\}$. For these $n_{i_{j}},$ we have $%
n_{i_{j+1}}-n_{i_{j}}<\varepsilon \cdot n_{i_{j}}<\varepsilon \cdot k^{L}$,
which implies that $\lbrack k^{L-1},k^{L}-1]\subset\bigcup_{n\in%
\mathcal{A}}[n-\varepsilon \cdot k^{L},n+\varepsilon \cdot k^{L}]$. Hence,
by Lemma \ref{lemma4} and (\ref{formula9}), there exists positive
constant $C_{1}(L)$ and positive integer $m_{1}(L)$ such that for any $%
m^{^{\prime }}\geq m_{1}(L)$, we have
\begin{equation*}
V_{L,m^{^{\prime }}}<V_{L,m^{^{\prime }}+1}\leq C_{1}(L)V_{L,m^{^{\prime
}}}^{1+\frac{(k+1)\varepsilon \cdot k^{L}}{k^{L-1}}}=C_{1}(L)V_{L,m^{^{%
\prime }}}^{1+\varepsilon \cdot k(k+1)}.
\end{equation*}

Let $U_{L,m^{^{\prime }}}$ be the integer $p_{l,m}$ such $V_{L,m^{^{\prime
}}}=q_{l,m}$. Then, by (\ref{formula10}), $\exists ~C_{2}(L),C_{3}(L)$ such
that
\begin{equation*}
\frac{C_{2}(L)}{V_{L,m^{^{\prime }}}^{\rho _{L}}}\leq \left\vert f(\frac{1}{b%
})-\frac{U_{L,m^{^{\prime }}}}{V_{L,m^{^{\prime }}}}\right\vert \leq \frac{%
C_{3}(L)}{V_{L,m^{^{\prime }}}^{\delta _{L}}},
\end{equation*}%
where $\rho _{L}$ is the value of $\rho _{l}$ at $l=k^{L}$, $\delta _{L}$ is
the value of $\delta _{l}$ at $l=k^{L-1}$. Thus, by Lemma \ref{lemma3}, for any sufficiently large $L$, we
have
\begin{equation*}
\mu (f(\frac{1}{b}))\leq \frac{\rho _{L}}{\delta _{L}-1}\cdot (1+\varepsilon
\cdot k(k+1)).
\end{equation*}

Note that $\lim\limits_{L\rightarrow \infty }\rho _{L}=\lim\limits_{L\rightarrow \infty
}\delta _{L}=2$. Hence, when $L\rightarrow \infty ,$ we have
\begin{equation*}
\mu (f(\frac{1}{b}))\leq 2\cdot (1+\varepsilon \cdot k(k+1)).
\end{equation*}

Since $\varepsilon $ is chosen arbitrarily, we are done.
\end{proof}

\end{document}